\documentclass[twoside]{article}

\usepackage[accepted]{aistats2020}
\usepackage[round]{natbib}

\usepackage{amssymb}
\usepackage{amsmath}
\usepackage[linesnumbered,boxed,ruled,vlined]{algorithm2e}
\usepackage{algorithmic}
\SetKwRepeat{DoWhile}{do}{while}

\newcommand{\eat}[1]{}

\newcommand{\R}{{\mathbb R}}

\newcommand{\ns}[1]{\| #1 \|^2}
\newcommand{\n}[1]{\| #1 \|}

{\hspace*{\fill}$\Box$\par\vspace{4mm}}
\newenvironment{proofof}[1]{\smallskip\noindent{\em Proof of #1.}}%
{\hspace*{\fill}$\Box$\par}
\newenvironment{proof}{\vspace{-0.05in}\noindent{\em Proof.}}%
{\hspace*{\fill}$\Box$\par}

\newcommand{\Am}{\mathsf{AA}}
\newcommand{\Gm}{\mathsf{GMRES}}
\newcommand*{\QEDB}{\hfill\ensuremath{\square}}
\newtheorem{theorem}{Theorem}

\newtheorem{lemma}{Lemma}
\newtheorem{corollary}{Corollary}
\newtheorem{definition}{Definition}
\newtheorem{assumption}{Assumption}

\usepackage{graphicx}

\usepackage{hyperref}
\hypersetup{
	colorlinks=true,       
	linkcolor=black,    
	citecolor=black,        
	filecolor=black,     
	urlcolor=black
}

\begin{document}

%
\runningtitle{A Fast Anderson-Chebyshev Acceleration for Nonlinear Optimization}

%
\runningauthor{Zhize Li, Jian Li}

\twocolumn[

\aistatstitle{A Fast Anderson-Chebyshev Acceleration for \\ Nonlinear Optimization}

\aistatsauthor{ Zhize Li \And Jian Li }
\aistatsaddress{ 
	        King Abdullah University of Science and Technology
					 \And    
					Tsinghua University } ]

\begin{abstract}
  \emph{Anderson acceleration} (or Anderson mixing) is an efficient acceleration method for fixed point iterations $x_{t+1}=G(x_t)$, e.g., gradient descent can be viewed as iteratively applying the operation $G(x) \triangleq x-\alpha\nabla f(x)$.
  It is known that Anderson acceleration is quite efficient in practice and can be viewed as an extension of Krylov subspace methods for nonlinear problems.
  In this paper, we show that Anderson acceleration with Chebyshev polynomial can achieve the optimal convergence rate $O(\sqrt{\kappa}\ln\frac{1}{\epsilon})$, which improves the previous result $O(\kappa\ln\frac{1}{\epsilon})$ provided by \citep{toth2015convergence} for quadratic functions.
  Moreover, we provide a convergence analysis for minimizing general nonlinear problems.
  Besides, if the hyperparameters (e.g., the Lipschitz smooth parameter $L$) are not available, we propose a \emph{guessing algorithm} for guessing them dynamically and also prove a similar convergence rate.
  Finally, the experimental results demonstrate that the proposed Anderson-Chebyshev acceleration method converges significantly faster than other algorithms, e.g., vanilla gradient descent (GD), Nesterov's Accelerated GD.
  Also, these algorithms combined with the proposed guessing algorithm (guessing the hyperparameters dynamically) achieve much better performance.
\end{abstract}

\section{Introduction}
Machine learning problems are usually modeled as optimization problems, ranging from convex optimization to highly nonconvex optimization such as deep neural networks, e.g.,  \citep{nesterov2014introductory, bubeck2015convex, lecun2015deep, lei2017non, li2018simple, fang2018spider, zhou2018stochastic, li2019stochastic, ge2019stable, li2019ssrgd}. 
To solve an optimization problem $\min_x f(x)$, the classical method is gradient descent, i.e.,  $x_{t+1}= x_t-\alpha_t\nabla f(x_t)$.
There exist several techniques to accelerate the standard gradient descent, e.g., momentum \citep{nesterov2014introductory, allen2017katyusha, lan2018random, lan2019unified}.
There are also various vector sequence acceleration methods developed
in the numerical analysis literature, e.g., \citep{brezinski2000convergence,sidi1986acceleration,smith1987extrapolation,brezinski1991extrapolation,brezinski2018shanks}.
Roughly speaking, if a vector sequence converges very slowly to its limit,
then one may apply such methods to accelerate the convergence of this sequence.
Taking gradient descent as an example, the vector sequence are generated by $x_{t+1}= G(x_t) \triangleq x_t-\alpha_t\nabla f(x_t)$, where the limit is the fixed-point $G(x^*) = x^*$ (i.e. $\nabla f(x^*)=0)$. One notable advantage of such acceleration methods is that they usually do not require to know how the vector sequence is actually generated. Thus the applicability of those methods is very wide.

Recently, \cite{scieur2016regularized} used the minimal polynomial extrapolation (MPE) method \citep{smith1987extrapolation} for convergence acceleration.
This is a nice example of using sequence acceleration methods to optimization problems.
In this paper, we are interested in another classical sequence acceleration method called \emph{Anderson acceleration} (or \emph{Anderson mixing}), which was proposed by Anderson in 1965 \citep{anderson1965iterative}.
The method is known to be quite efficient in a variety of applications
\citep{capehart1989techniques,pratapa2016anderson,higham2016anderson,loffeld2016considerations}.
The idea of Anderson acceleration is to maintain $m$ recent iterations for determining the next iteration point, where $m$ is a parameter (typically a very small constant).
Thus, it can be viewed as an extension of the existing momentum methods which usually use the last and current points to determine the next iteration point.
Anderson acceleration with slight modifications is described in Algorithm \ref{alg:am}.

\begin{algorithm}[h]
	\caption{Anderson Acceleration($m$)}
	\label{alg:am}
	\textbf{input:} $x_0, T, \lambda, \beta_t$\\
	Define $G(x)\triangleq x+F \triangleq x - \lambda\nabla f(x)$\;
	$x_1=G(x_0)$, $F_0=G(x_0)-x_0$\;
	\For{$t= 1, 2,\ldots T$}{
		$m_t=\min\{m,t\}$\;
		$F_t\triangleq G(x_t)-x_t$\;
		Solve $\min_{\alpha^t=(\alpha_0^t,\ldots,\alpha_{m_t}^t)^T}
		\left\|\sum_{i=0}^{m_t}{\alpha_i^t F_{t-i}}\right\|_2$  subject to $\sum_{i=0}^{m_t}{\alpha_i^t=1}$\; \label{line:ls}
		$x_{t+1}=(1-\beta_t)\sum_{i=0}^{m_t}{\alpha_i^tx_{t-i}} + \beta_t\sum_{i=0}^{m_t}{\alpha_i^tG(x_{t-i})}$\;
	}
	\Return $x_T$
\end{algorithm}

Note that the step in Line \ref{line:ls} of Algorithm \ref{alg:am} can be transformed to an equivalent unconstrained least-squares problem:
\begin{equation}\label{eq:transt}
\min_{(\alpha_1^t,\ldots,\alpha_{m_t}^t)^T}\Big\|F_t-\sum_{i=1}^{m_t}{\alpha_i^t (F_t-F_{t-i})}\Big\|_2,
\end{equation}
then let $\alpha_0^t= 1-\sum_{i=1}^{m_t}{\alpha_i^t}$. Using QR decomposition, \eqref{eq:transt} can be solved in time $2m_t^2d$, where $d$ is the dimension. Moreover, the QR decomposition of \eqref{eq:transt} at iteration $t$ can be efficiently obtained from that of at iteration $t-1$ in $O(m_td)$
(see, e.g. \citep{golub1996matrix}). The constant $m_t\leq m$ is usually very small. We use $m=3$ and $5$ for the numerical experiments in Section \ref{sec:exp}.
Hence, each iteration of Anderson acceleration can be implemented quite efficiently.

Many studies showed the relations between Anderson acceleration and other optimization
methods. 
In particular, for the quadratic case (linear problems), \cite{walker2011anderson} showed that it is related to the well-known Krylov subspace method GMRES (generalized minimal residual algorithm) \citep{saad1986gmres}.
Furthermore, \cite{potra2013characterization} showed that GMRES is equivalent to Anderson acceleration with any mixing parameters under $m =\infty$ (see Line 5 of Algorithm \ref{alg:am})
for linear problems.
Concretely, \cite{toth2015convergence} proved the first linear convergence rate $O(\kappa\ln\frac{1}{\epsilon})$ for linear problems with fixed parameter $\beta$, where $\kappa$
is the condition number.
Besides, \cite{eyert1996comparative}, and \cite{fang2009two} showed that Anderson acceleration is related to the multisecant quasi-Newton methods (more concretely, the generalized Broyden's second method).
Despite the above results,
the convergence results for this efficient method are still limited
(especially for general nonlinear problems and the case where $m$ is small).
In this paper, we analyze the convergence for small $m$ which is the typical case in practice and also provide the convergence analysis for general nonlinear problems.

\subsection{Our Contributions}

There has been a growing number of applications of Anderson acceleration method
\citep{pratapa2016anderson,higham2016anderson,loffeld2016considerations,scieur2018nonlinear}.
Towards a better understanding of this efficient method,
we make the following technical contributions:
\begin{enumerate}
	\item We prove the optimal $O(\sqrt{\kappa}\ln\frac{1}{\epsilon})$ convergence rate of the proposed Anderson-Chebyshev acceleration (i.e., Anderson acceleration with Chebyshev polynomial) for minimizing quadratic functions (see Theorem \ref{thm:opt}).
	Our result improves the previous result $O(\kappa\ln\frac{1}{\epsilon})$ 
	given by \citep{toth2015convergence} and matches the lower bound $\Omega(\sqrt{\kappa}\ln\frac{1}{\epsilon})$ provided by \citep{nesterov2014introductory}.
	Note that for ill-conditioned problems, the condition number $\kappa$ can be very large.
	
	\item Then, we prove the linear-quadratic convergence
	of Anderson acceleration for minimizing general nonlinear problems
	under some standard assumptions (see Theorem \ref{thm:gel}).
	Compared with Newton-like methods, it is more attractive
	since it does not require to compute (or approximate) Hessians,
	or Hessian-vector products.
	
	\item Besides,
	we propose a \emph{guessing algorithm} (Algorithm \ref{alg:guess}) for the case when the hyperparameters (e.g., $\mu, L$) are not available.
	We prove that it achieves a similar convergence rate $O(\sqrt{\kappa}\ln\frac{1}{\epsilon}+\sqrt{\kappa}(\ln\kappa\ln B)^2)$ (see Theorem \ref{thm:guess}). This guessing algorithm can also be combined with other algorithms, e.g., Gradient Descent (GD), Nesterov's Accelerated GD (NAGD).
	The experimental results (see Section \ref{app:expga}) show that these algorithms combined with the proposed guessing algorithm achieve much better performance. 
	
	\item Finally, the experimental results on the real-world UCI datasets and synthetic datasets demonstrate that Anderson acceleration methods converge significantly faster than other algorithms (see Section \ref{sec:exp}). Combined with our theoretical results, the experiments validate that Anderson acceleration methods (especially Anderson-Chebyshev acceleration) are efficient both in theory and practice.
\end{enumerate}

\subsection{Related Work}
\label{sec:rw}
As aforementioned, Anderson acceleration can be viewed as the extension of the momentum methods (e.g., NAGD) and the potential extension of Krylov subspace methods (e.g., GMRES) for nonlinear problems.
In particular, GD is the special case of Anderson acceleration with $m=0$, and to some extent NAGD can be viewed as $m=1$.
We also review the equivalence of GMRES and Anderson acceleration without truncation (i.e., $m =\infty$) in Appendix \ref{app:gam}. 
Besides, \cite{eyert1996comparative}, and \cite{fang2009two} showed that Anderson acceleration is related to the multisecant quasi-Newton methods.
Note that Anderson acceleration has the advantage over the Newton-like methods since it does not require the computation of Hessians or approximation of Hessians or Hessian-vector products.

There are many sequence acceleration methods in the numerical analysis literatures. In particular, the well-known Aitken's $\Delta^2$ process \citep{aitken1926bernoulli} accelerated the convergence of a sequence that is converging linearly.
Shanks generalized the Aitken extrapolation which was known as Shanks transformation \citep{shanks1955non}.
Recently, \cite{brezinski2018shanks} proposed a general framework for Shanks sequence transformations which includes many vector sequence acceleration methods.
One fundamental difference between Anderson acceleration and other sequence acceleration methods (such as MPE, RRE (reduced rank extrapolation) \citep{sidi1986acceleration,smith1987extrapolation}, etc.) is
that Anderson acceleration is a fully dynamic method \citep{capehart1989techniques}.
Here \emph{dynamic} means all iterations are in the same sequence, and it does not require to restart the procedure. It can be seen from Algorithm \ref{alg:am} that
all iterations are applied to the same sequence $\{x_t\}$.
In fact, in Capehart's PhD thesis \citep{capehart1989techniques}, several experiments were conducted to demonstrate the superior performance of Anderson acceleration over other semi-dynamic methods
such as MPE, RRE (semi-dynamic means that the algorithm maintains more than one sequences or
needs to restart several times).
More recently, Anderson acceleration with different variants and/or under different assumptions are widely studied (see e.g.,  \citep{zhang2018globally, evans2018proof, scieur2018nonlinearv2}).

\section{The Quadratic Case}
\label{sec:opt}

In this section, we consider the problem of minimizing a quadratic function
(also called least squares, or ridge regression \citep{boyd2004convex,hoerl1970ridge}).
The formulation of the problem is
\begin{equation}\label{prob:quad}
\min_{x\in \R^d} f(x)=\frac{1}{2}x^TAx-b^Tx,
\end{equation}
where $\mu I_d\preceq\nabla^2f=A\preceq LI_d$.
Note that $\mu$ and $L$ are usually called the strongly convex parameter and Lipschitz continuous gradient parameter, respectively (e.g. \citep{nesterov2014introductory, allen2017katyusha, lan2019unified}).
There are many algorithms for optimizing this type of functions. See e.g.~\citep{bubeck2015convex} for more details.
We analyze the problem of minimizing a more general function $f(x)$ in the next Section \ref{sec:gel}.

We prove that Anderson acceleration with Chebyshev polynomial parameters $\{\beta_t\}$ achieves the optimal convergence rate, i.e., it obtains an $\epsilon$-approximate solution using $O(\sqrt{\kappa}\ln\frac{1}{\epsilon})$ iterations.
The convergence result is stated in the following Theorem \ref{thm:opt}.

\begin{theorem}\label{thm:opt}
	The Anderson-Chebyshev acceleration method achieves the optimal convergence rate $O(\sqrt{\kappa}\ln\frac{1}{\epsilon})$ for obtaining an $\epsilon$-approximate solution of problem (\ref{prob:quad}) for any $0\leq m\leq k$,
	where $\kappa=L/\mu$ is the condition number, $k$ is defined in Definition \ref{def:k} and
	this method combines Anderson acceleration (Algorithm \ref{alg:am}) with
	the Chebyshev polynomial parameters
	$\beta_t=1/\big(\frac{L+\mu}{2}+\frac{L-\mu}{2}\cos(\frac{(2t-1)\pi}{2T})\big)$,
	for $t=1, 2, \ldots, T$.
\end{theorem}
\textbf{Remark:} In this quadratic case, we mention that \cite{toth2015convergence} proved the first convergence rate $O(\kappa\ln\frac{1}{\epsilon})$ for fixed parameter $\beta$.
Here we use the Chebyshev polynomials to improve the result to the optimal $O(\sqrt{\kappa}\ln\frac{1}{\epsilon})$ which matches the lower bound $\Omega(\sqrt{\kappa}\ln\frac{1}{\epsilon})$.
Note that for ill-conditioned problems, the condition number $\kappa$ can be very large.
Also note that in practice the constant $m$ is usually very small. Particularly, $m=3$ has already achieved a remarkable performance from our experimental results (see Figures \ref{fig:lreg}--\ref{fig:5} in Section \ref{sec:exp}).

Before proving Theorem \ref{thm:opt}, we first define $k$ and then briefly review some properties of the Chebyshev polynomials.
We refer to \citep{rivlin1974chebyshev,olshanskii2014iterative,hageman2012applied} for more details of Chebyshev polynomials.

\begin{definition}\label{def:k}
	Let $v_i$'s be the unit eigenvectors of $A$, where $A$ is defined in \eqref{prob:quad}. Consider a unit vector $c\triangleq \sum_{i=1}^d{c_iv_i}$ and let $c'\triangleq \mathrm{Proj}_{B_k^\perp} c=\sum_{i=1}^d{c_i'v_i}$, where
	$\mathrm{Proj}_{B_k^\perp}$ denotes the projection to the orthogonal complement of the column space of $B_k \triangleq A[x_{t-k}-x_{t}, \dotsc, x_{t-1}-x_{t}] \in \mathbb{R}^{d\times k}$.
	Define $k$ to be the maximum integer such that $c_i'\leq (1+\frac{1}{\sqrt{\kappa}+1}){c_i}$ for any $i\in [d]$. 
\end{definition}
Obviously, $k\geq 0$ since $c'=c$ due to $B_0=0$ and $\mathrm{Proj}_{B_0^\perp}=I$.

Now we review the Chebyshev polynomials.  The \emph{Chebyshev polynomials} are polynomials $P_k(x)$, where $k\geq 0$, $\deg(P_k)=k$, which is defined by the recursive relation:
\begin{equation}\label{eq:cheb0}
\begin{split}
& P_0(x)=1,\\
& P_1(x)=x, \\
& P_{k+1}(x)=2xP_k(x)-P_{k-1}(x).
\end{split}
\end{equation}
The key property is that $P_k(x)$ has minimal deviation from $0$ on $[-1,1]$ among all polynomials $Q_k$ with $\deg(Q_k)=k$ and leading coefficient $\alpha_k=2^{k-1}$ for the largest degree term $x^k$, i.e.,
\begin{equation}\label{eq:chebmin}
\max_{x\in[-1,1]}|P_k(x)|\leq \max_{x\in[-1,1]}|Q_k(x)| \quad \mathrm{for~all}~ Q_k.
\end{equation}
In particular, for $|x|\leq 1$, Chebyshev polynomials can be written in an equivalent way:
\begin{align}
P_k(x)=\cos(k\arccos x). \label{eq:cheb}
\end{align}
In our proof, we use this equivalent form (\ref{eq:cheb}) instead of  (\ref{eq:cheb0}).
The equivalence can be verified as follows:
\begin{align}
P_k(x)&=2x\cos((k-1)\arccos x)-\cos((k-2)\arccos x) \notag\\
&= 2\cos \theta \cos((k-1)\theta)-\cos((k-2)\theta) \label{eq:xx11}\\
&=\cos(k\theta)+\cos((k-2)\theta)-\cos((k-2)\theta) \notag\\
&=\cos(k\arccos x), \label{eq:xx22}
\end{align}
where (\ref{eq:xx11}) and (\ref{eq:xx22}) use the transformation $x=\cos\theta$ due to $|x|\leq 1$.
According to (\ref{eq:cheb}), $\max_{x\in[-1,1]}|P_k(x)|=1$ and the $k$ roots of $P_k$ are as follows:
\begin{equation}\label{eq:root}
x_i=\cos\Big(\frac{(2i-1)\pi}{2k}\Big),~ i=1,2,\ldots,k.
\end{equation}
To demonstrate it more clearly, we provide an example for $P_4(x)$ (W-shape curve) in Figure \ref{fig:cheb}.
Since $k=4$ in this polynomial $P_4(x)$, the first root $x_1=\cos\left(\frac{(2i-1)\pi}{2k}\right)=\cos\left(\frac{\pi}{8}\right)\thickapprox 0.92$. The remaining three roots for $P_4(x)$ can be easily computed too.
\begin{figure}[!h]
	\centering
	\includegraphics[width=\linewidth]{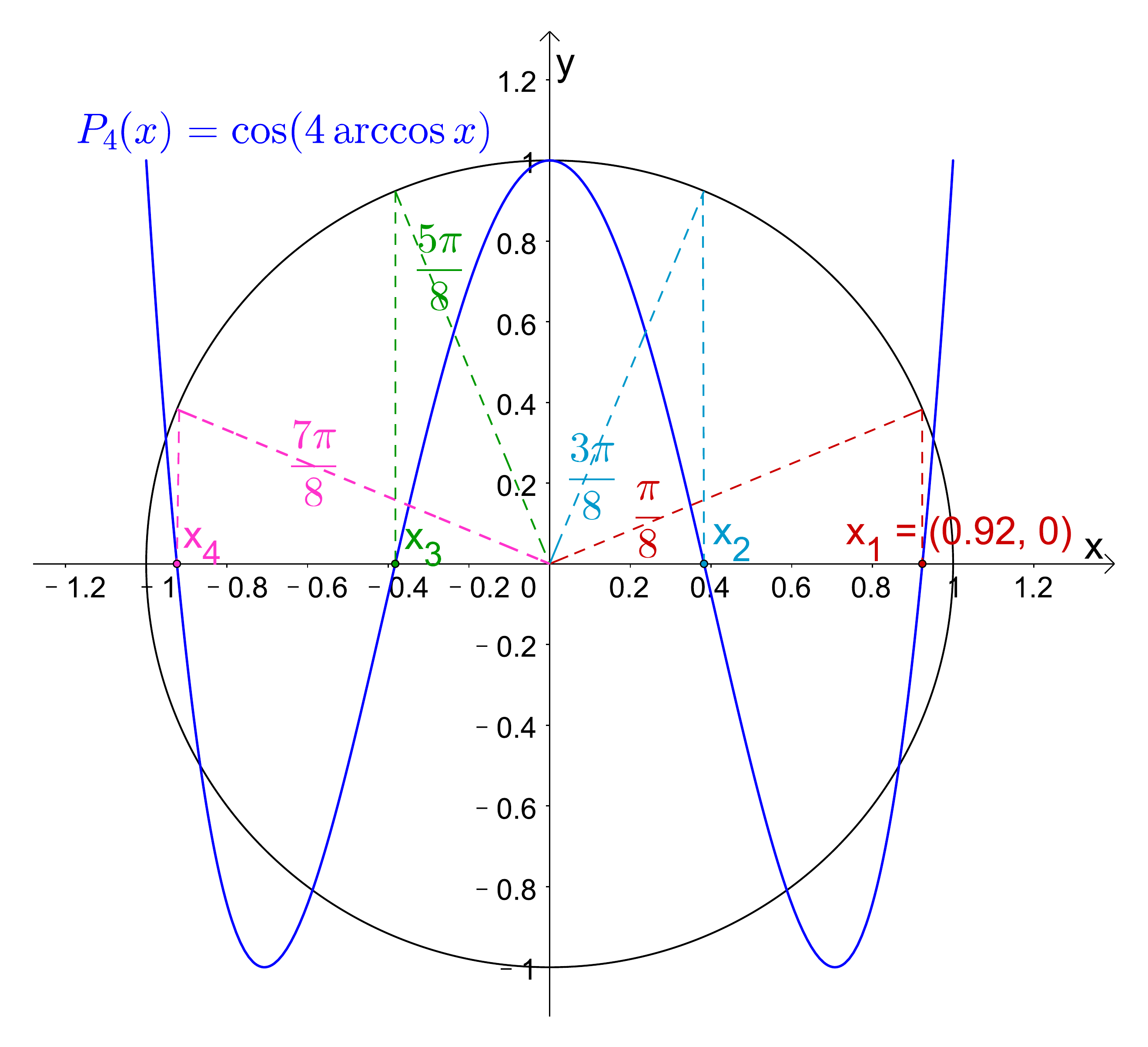}
	\caption{The Chebyshev polynomial $P_4(x)$}
	\label{fig:cheb}
\end{figure}

\begin{proofof}{Theorem \ref{thm:opt}}
	For iteration $t+1$, the residual $F_{t+1}\triangleq -\lambda\nabla f(x_{t+1}) = -(Ax_{t+1}-b)$  (let $\lambda=1$) can be deduced as follows:
	\begin{align}
	F_{t+1}
	&=b-Ax_{t+1}\notag \\
	&=b-A\biggl[(1-\beta_t)\sum_{i=0}^{m_t}{\alpha_i^t x_{t-i}}+
	\beta_t\sum_{i=0}^{m_t}{\alpha_i^t G(x_{t-i})}\biggr]\notag \\
	&=b-A\biggl[\sum_{i=0}^{m_t}{\alpha_i^t x_{t-i}}+
	\beta_t\sum_{i=0}^{m_t}{\alpha_i^t \left(b-Ax_{t-i})\right)}\biggr] \label{eq:quadx1}\\
	&=b-\beta_tAb-A\biggl[
	\sum_{i=0}^{m_t}{\alpha_i^t \left((I-\beta_tA)x_{t-i})\right)}\biggr] \notag \\
	&=(I-\beta_tA)\sum_{i=0}^{m_t}{\alpha_i^t \left(b-Ax_{t-i})\right)} \notag \\
	&=(I-\beta_tA)\sum_{i=0}^{m_t}{\alpha_i^t F_{t-i}},  \label{eq:quadx2}
	\end{align}
	where (\ref{eq:quadx1}) uses $G(x_t)=x_t +F_t$.
	
	To bound $\|F_{t+1}\|_2$ (i.e., $\|\nabla f(x_{t+1})\|_2$), we first obtain the following lemma by using Singular Value Decomposition (SVD) to solve the least squares problem (\ref{eq:transt}) and then using several transformations. We defer the proof of Lemma \ref{lm:boundft1} to Appendix \ref{app:pflm2}.
	\begin{lemma}\label{lm:boundft1}
		Let $F_1=b-Ax_{1}$ and $F_{t+1}=b-Ax_{t+1}$, then
		\begin{equation}\label{eq:chebh}
		\|F_{t+1}\|_2/\|F_1\|_2 \leq \sqrt{2\min_{\beta}\max_{\lambda\in [\mu,L]}|H_t(\lambda)|}
		\end{equation}
		where $H_t(\lambda) =(1-\beta_t\lambda)\dotsm(1-\beta_1\lambda)$ is a degree $t$ polynomial.
	\end{lemma}
	
	According to Lemma \ref{lm:boundft1}, to bound $\|F_{t+1}\|_2$, it is sufficient to bound the right-hand-side (RHS) of (\ref{eq:chebh}) (i.e., $\min_{\beta}\max_{\lambda\in [\mu,L]}|H_t(\lambda)|$). 
	So we want to choose parameter $\beta$ in order to make $\max_{\lambda\in [\mu,L]}|H_t(\lambda)|$ as small as possible. 
	According to (\ref{eq:chebmin}) (the minimal deviation property of standard Chebyshev polynomials), 
	hence a natural idea is to choose $\beta$ such that $H_t(\lambda) =(1-\beta_t\lambda)\dotsm(1-\beta_1\lambda)$ is a kind of modified Chebyshev polynomials.
	In order to do this, we first transform $[\mu,L]$ into $[-1,1]$, i.e., let $\lambda=\frac{L+\mu}{2}+\frac{L-\mu}{2}x$, where $x\in[-1,1]$. 
	Also note that polynomial $H_t(\lambda) =(1-\beta_t\lambda)\dotsm(1-\beta_1\lambda)$ has (only) one constraint, i.e., $H_t(0)=1$. Thus we choose $\beta$ such that 
	\begin{align}
	H_t(\lambda)&=P_t\Big(\frac{2\lambda-(L+\mu)}{L-\mu}\Big)\Big/P_t\Big(-\frac{L+\mu}{L-\mu}\Big) \notag\\
	&=P_t(x)\big{/}P_t\Big(-\frac{L+\mu}{L-\mu}\Big), \label{eq:trh}
	\end{align}
	where $P_t(\cdot)$ is the standard Chebyshev polynomials.
	Now, the RHS of (\ref{eq:chebh}) can be bounded as follows:
	\begin{align}
	&\min_{\beta}\max_{\lambda\in [\mu,L]}|H_t(\lambda)| \notag\\
	&\leq \max_{x\in [-1,1]}\Big|P_t(x)\big{/}P_t\Big(-\frac{L+\mu}{L-\mu}\Big)\Big| \label{eq:trhcall}\\
	&\leq 1\Big{/}\Big|P_t\Big(-\frac{L+\mu}{L-\mu}\Big)\Big|, \label{eq:ptx}
	\end{align}
	where (\ref{eq:trhcall}) uses (\ref{eq:trh}), and (\ref{eq:ptx}) uses $\max_{x\in[-1,1]}|P_t(x)|=1$ (see (\ref{eq:cheb})).
	According to (\ref{eq:root}), it is not hard to see that $H_t(\lambda)$ is defined by the mixing parameters $\beta_i=1\Big/\Big(\frac{L+\mu}{2}+\frac{L-\mu}{2}\cos\big(\frac{(2i-1)\pi}{2t}\big)\Big)$ according to $\lambda=\frac{L+\mu}{2}+\frac{L-\mu}{2}x$, where $i=1,2,\ldots,t$. Note that the roots of standard Chebyshev polynomials (i.e., (\ref{eq:root}))  can be found from many textbooks, e.g., Section 1.2 of \citep{rivlin1974chebyshev}.
	Now, we only need to bound $\big|P_t\big(-\frac{L+\mu}{L-\mu}\big)\big|$. First, we need to transform the form (\ref{eq:cheb}) of Chebyshev polynomials $P_t(x)$ as follows:
	\begin{align*}
	P_t(x)&=\cos(t\arccos x)\\
	&=\cos(t\theta) \qquad \mathrm{Define}~x\triangleq\cos\theta\\
	&=\left(e^{i\theta t}+e^{-i\theta t}\right)/2 \\
	&=\left((\cos\theta+i\sin\theta)^t+(\cos\theta-i\sin\theta)^t\right)/2\\
	&=\Big(\big(x+\sqrt{x^2-1}\big)^t+\big(x-\sqrt{x^2-1}\big)^t\Big)/2.
	\end{align*}
	Let $x=-\frac{L+\mu}{L-\mu}$, we get $\sqrt{x^2-1}=\sqrt{\frac{(L+\mu)^2-(L-\mu)^2}{(L-\mu)^2}}=\sqrt{\frac{4L\mu}{(L-\mu)^2}}
	=\frac{2\sqrt{L\mu}}{L-\mu}$. So we have
	\begin{align}
	\Big|P_t\Big(-\frac{L+\mu}{L-\mu}\Big)\Big|
	&\geq \frac{1}{2}\Big(\frac{L+\mu}{L-\mu}+\frac{2\sqrt{L\mu}}{L-\mu}\Big)^t \notag\\
	&\geq \frac{1}{2}\Big(\frac{\sqrt{L}+\sqrt{\mu}}{\sqrt{L}-\sqrt{\mu}}\Big)^t. \label{eq:thx1}
	\end{align}
	Now, the RHS of (\ref{eq:chebh}) can be bounded as
	\begin{align}
	\sqrt{2\min_{\beta}\max_{\lambda\in [\mu,L]}|H_t(\lambda)|} &\leq  \sqrt{2\big{/}\Big|P_t\Big(-\frac{L+\mu}{L-\mu}\Big)\Big|} \label{eq:thx2}\\
	& \leq 2\Big(\frac{\sqrt{L}-\sqrt{\mu}}{\sqrt{L}+\sqrt{\mu}}\Big)^{t/2}, \label{eq:thx3}
	\end{align}
	where (\ref{eq:thx2}) follows from (\ref{eq:ptx}), and (\ref{eq:thx3}) follows from (\ref{eq:thx1}).
	Then, according to (\ref{eq:chebh}), the gradient norm is bounded as
	$\|\nabla f(x_{t+1})\|_2=\|F_{t+1}\|_2
	\leq 2\big(\frac{\sqrt{L}-\sqrt{\mu}}{\sqrt{L}+\sqrt{\mu}}\big)^{t/2} \|F_1\|_2
	=2\big(\frac{\sqrt{\kappa}-1}{\sqrt{\kappa}+1}\big)^{t/2} \|\nabla f(x_1)\|_2$,
	where $\kappa=L/\mu$.
	Note that if the number of iterations $t=(\sqrt{\kappa}+1)\ln\frac{1}{\epsilon}$, then
	\begin{align*}
	\big(\frac{\sqrt{\kappa}-1}{\sqrt{\kappa}+1}\big)^{t/2}=\big(1-\frac{2}{\sqrt{\kappa}+1}\big)^{t/2} \leq \epsilon.
	\end{align*}
	Thus the Anderson-Chebyshev acceleration method achieves the optimal convergence rate $O(\sqrt{\kappa}\ln\frac{1}{\epsilon})$ for obtaining an $\epsilon$-approximate solution.
\end{proofof}

\section{The General Case}
\label{sec:gel}
In this section, we analyze the Anderson Acceleration (Algorithm \ref{alg:am}) in the general nonlinear case:
\begin{equation}\label{prob:gel}
\min_{x \in \R^d} f(x).
\end{equation}
We prove that Anderson acceleration method achieves the linear-quadratic convergence rate under the following standard Assumptions \ref{asp:1} and \ref{asp:2},
where $\|\cdot\|$ denotes the Euclidean norm.
Let $\mathcal{B}_t$ denote the small matrix of the least-square problem in Line 7 of Algorithm \ref{alg:am},
i.e., $\mathcal{B}_t\triangleq[F_{t}-F_{t-1}, \dotsc, F_t-F_{t-m}] \in \mathbb{R}^{d\times m}$ (see problem (\ref{eq:transt})).
Then, we define its condition number $\kappa_t\triangleq\|\nabla f(x_t)\|/\tilde{\mu}_t$ and $\tilde{\kappa} \triangleq \max_t\{\kappa_t\}$, where $\tilde{\mu}_t$ denotes the least non-zero singular value of $\mathcal{B}_t$.

\begin{assumption}\label{asp:1}
	The Hessian $\nabla^2f$ satisfies $\mu\leq \|\nabla^2f\| \leq L$,
	where $0\leq \mu \leq L$.
\end{assumption}
\begin{assumption}\label{asp:2}
	The Hessian $\nabla^2f$ is $\gamma$-Lipschitz continuous, i.e.,
	\begin{equation}\label{asp3}
	\|\nabla^2f(x) - \nabla^2f(y)\| \leq \gamma \|x-y\|.
	\end{equation}
\end{assumption}

\begin{theorem}\label{thm:gel}
	Suppose Assumption \ref{asp:1} and \ref{asp:2} hold. Let step-size $\lambda=\frac{2}{L+\mu}$.
	The convergence rate of Anderson Acceleration($m$) (Algorithm \ref{alg:am}) is linear-quadratic for problem (\ref{prob:gel}), i.e.,
	\begin{equation}\label{eq:converge}
	\|\nabla f(x_{t+1})\| \leq c_1\Delta_t^2 + c_2\Delta_t\|\nabla f(x_t)\| + (1-c_3)\|\nabla f(x_t)\|,
	\end{equation}
	where $c_1=\frac{3\tilde{\kappa}^2\gamma m}{(L+\mu)^2}$, ~$c_2=\frac{2\tilde{\kappa}\beta_t\gamma\sqrt{m}}{(L+\mu)^2}$, ~$c_3=\beta_t\frac{2\mu}{L+\mu}$ and $\Delta_t \triangleq \max_{i\in[m]}\|x_t-x_{t-i}\|$.
\end{theorem}

\noindent\textbf{Remark:}\vspace{-3mm}
\begin{enumerate}
	\item The constant $m\geq 0$ is usually very small. Particularly, we use $m=3$ and $5$ for the numerical experiments in Section \ref{sec:exp}. Hence $\Delta_t$ is very small and also decreases as the algorithm converges.
	\item Besides, one can also use $\|\nabla f(x_t)\|$ instead of $\Delta_t$ in  (\ref{eq:converge}) according to the property of $f$ (Assumption \ref{asp:1}), i.e., $\mu\|x_t-x^*\| \leq\|\nabla f(x_t)-\nabla f(x^*)\| =\|\nabla f(x_t)\|$, and
	$\|x_t-x_{t-i}\| = \|x_t-x^*+x^*-x_{t-i}\| 
	\leq \|x_t-x^*\|+\|x_{t-i}-x^*\|$.

	\item Note that the first two terms in RHS of (\ref{eq:converge}) converge quadratically and the last term converges linearly. 
	Due to the fully dynamic property of Anderson acceleration as we discussed in Section \ref{sec:rw}, it turns out the exact convergence rate of Anderson acceleration in the general case is not easy to obtain.
	But we note that the convergence rate is roughly linear, i.e., $O(\frac{1}{c_3}\log\frac{1}{\epsilon})$ since the first two quadratic terms converge much faster than the last linear term in some neighborhood of optimum.
	In particular, if $f$ is a quadratic function, then $\gamma =0$ (Assumption \ref{asp:2}) and thus $c_1=c_2=0$ in (\ref{eq:converge}). Only the last linear term remained, thus it converges linearly (see the following corollary).
\end{enumerate}
\vspace{-2mm}
\begin{corollary}\label{cor:quad}
	If $f$ is a
	quadratic function, let step-size $\lambda=\frac{2}{L+\mu}$ and $\beta_t=1$.
	Then the convergence rate of Anderson Acceleration is linear, i.e., $O(\kappa\ln\frac{1}{\epsilon})$, where $\kappa=L/\mu$ is the condition number.
\end{corollary} \vspace{-1mm}
Note that this corollary recovers the previous result (i.e., $O(\kappa\ln\frac{1}{\epsilon})$) obtained by \citep{toth2015convergence},
and we use \emph{Chebyshev polyniomial} to improve this result to the optimal convergence rate $O(\sqrt{\kappa}\ln\frac{1}{\epsilon})$ in our previous Section \ref{sec:opt} (see Theorem \ref{thm:opt}).
Concretely, we transfer the weight of step-size $\lambda$ to the parameters $\beta_t$'s and use Chebyshev polynomial parameters $\beta_t$'s in our Theorem \ref{thm:opt} instead of using fixed parameter $\beta \equiv 1$. 

\vspace{1mm}
Now, we provide a proof sketch for Theorem \ref{thm:gel}. The detailed proof can be found in Appendix \ref{app:pfthm1}.

\noindent{\em Proof Sketch of Theorem \ref{thm:gel}.}
Consider the iteration $t+1$, we have $F_t=-\frac{2}{L+\mu}\nabla f(x_t)$ according to $\lambda =\frac{2}{L+\mu}$.
First, we need to demonstrate several useful forms of $x_{t+1}$ as follows:
\begin{align}
x_{t+1}&= (1-\beta_t)\sum_{i=0}^{m_t}{\alpha_i^t x_{t-i}}
+ \beta_t\sum_{i=0}^{m_t}{\alpha_i^t G(x_{t-i})} \notag \\
&=\sum_{i=0}^{m_t}{\alpha_i^t x_{t-i}}
+ \beta_t\sum_{i=0}^{m_t}{\alpha_i^t\Bigl(G(x_{t-i})-x_{t-i}\Bigr)} \notag \\
&=\sum_{i=0}^{m_t}{\alpha_i^t x_{t-i}}
+ \beta_t\sum_{i=0}^{m_t}{\alpha_i^tF_{t-i}}  \label{eq:x3}\\
&=x_t-\sum_{i=1}^{m_t}{\alpha_i^t (x_t-x_{t-i})}  \notag\\
 &\qquad\quad+ \beta_t\Big(F_t-\sum_{i=1}^{m_t}{\alpha_i^t(F_t-F_{t-i})}\Big), \label{eq:x2}
\end{align}
where (\ref{eq:x3}) holds due to the definition $G_t=G(x_t)=x_t+F_t$, and (\ref{eq:x2}) holds since $\sum_{i=0}^{m_t}{\alpha_i^t=1}$.

Then, to bound $\|F_{t+1}\|_2$ (i.e., $\|\nabla f(x_{t+1})\|_2$), we deduce $F_{t+1}$ as follows:
\begin{align}
F_{t+1}
&= G_{t+1} - x_{t+1} \notag \\
&= G_{t+1} - \sum_{i=0}^{m_t}{\alpha_i^t x_{t-i}}
- \beta_t\sum_{i=0}^{m_t}{\alpha_i^tF_{t-i}} \notag \\
&= G_{t+1} - \sum_{i=0}^{m_t}{\alpha_i^t (G_{t-i}-F_{t-i})}
- \beta_t\sum_{i=0}^{m_t}{\alpha_i^tF_{t-i}} \notag \\
&= G_{t+1} - \sum_{i=0}^{m_t}{\alpha_i^t G_{t-i}} + (1-\beta_t)\mathcal{F}, \label{eq:sf2}
\end{align}
where (\ref{eq:sf2}) uses the definition
$\mathcal{F} \triangleq \sum_{i=0}^{m_t}{\alpha_i^tF_{t-i}}$.
Now, we bound the first two terms of (\ref{eq:sf2}) as follows:
\begin{align}
&G_{t+1} - \sum_{i=0}^{m_t}{\alpha_i^t G_{t-i}}  \notag\\
&= G_{t+1} - \bigl(G_{t} - \sum_{i=1}^{m_t}{\alpha_i^t}(G_{t}-G_{t-i})\bigr) \notag\\
&= \int_0^1G'\Big(x_t+u(x_{t+1}-x_t)\Big)(x_{t+1}-x_t)\,du \notag\\
&\quad~
-\sum_{i=1}^{m_t}{\alpha_i^t}\int_0^1G'\Bigl(x_t+u(x_{t-i}-x_t)\Bigr)(x_{t-i}-x_t)\,du \notag\\
&= \sum_{i=1}^{m_t}{\alpha_i^t}\int_0^1G'\Bigl(x_t+u(x_{t+1}-x_t)\Bigr)(x_{t-i}-x_t)\,du \notag\\
&\quad~
+ \int_0^1G'\Bigl(x_t+u(x_{t+1}-x_t)\Bigr)\beta_t\mathcal{F}\,du  \notag\\
&\quad~ -\sum_{i=1}^{m_t}{\alpha_i^t}\int_0^1G'\Bigl(x_t+u(x_{t-i}-x_t)\Bigr)(x_{t-i}-x_t)\,du,  \label{eq:sint}
\end{align}
where (\ref{eq:sint}) is obtained by using (\ref{eq:x2}) to replace $x_{t+1}$.
To bound (\ref{eq:sint}), we use Assumptions \ref{asp:1}, \ref{asp:2}, and the equation
\begin{equation*}
G_t' = I+F_t' = I - \frac{2}{L+\mu}\nabla^2f(x_t).
\end{equation*}
After some non-trivial calculations (details can be found in Appendix \ref{app:pfthm1}), we obtain
\begin{align*}
\|F_{t+1}\| &\leq \frac{\gamma(m\ns{\alpha}+\sqrt{m}\n{\alpha})\Delta_t^2}{L+\mu} + \frac{\gamma\sqrt{m}\n{\alpha}\beta_t\Delta_t\|\mathcal{F}\|}{L+\mu} \\
& \quad\qquad +\Bigl(1-\frac{2\mu}{L+\mu}\beta_t\Bigr)\|\mathcal{F}\|,
\end{align*}
where $\n{\alpha}$ denotes the Euclidean norm of $\alpha=(\alpha_1^t,\ldots,\alpha_{m_t}^t)^T$.
Then, according to the problem (\ref{eq:transt}) and the definition of $\mathcal{F} \triangleq \sum_{i=0}^{m_t}{\alpha_i^tF_{t-i}}$, we have $\|\mathcal{F}\| \leq  \|F_t\|$.
Finally, we bound $\n{\alpha}\leq\frac{2\tilde{\kappa}}{L+\mu}$ using QR decomposition of problem (\ref{eq:transt}) and recall $F_t=-\frac{2}{L+\mu}\nabla f(x_t)$ to finish the proof of Theorem \ref{thm:gel}.
\QEDB

\section{Guessing Algorithm}
\label{sec:guess}
In this section, we provide a \emph{Guessing Algorithm} (described in Algorithm \ref{alg:guess}) which guesses the parameters (e.g., $\mu, L$) dynamically.
Intuitively, we guess the parameter $\mu$ and the condition number $\kappa$ in a doubling way.
Note that in general these parameters are not available, since the time for computing these parameters is almost the same as (or even longer than) solving the original problem.
Also note that the condition in Line 14 of Algorithm \ref{alg:guess} depends on the algorithm used in Line 12.

\begin{algorithm}[!h]
	\caption{Guessing Algorithm}
	\label{alg:guess}
	\textbf{input:} $x_0, T, \delta, B$\\
	Let $t=0$\;
	\For{$i= 1, 2,\ldots $}{
		$\kappa_i=e^{i+2}$\;
		\For{$j=1,2,\ldots, \ln B$}{
			$\mu_i=e^j\delta,  L_i = \mu_i\kappa_i, t_i = 1$\;
			\DoWhile{ $\frac{\|\nabla f(x_t)\|}{\|\nabla f(x_{t-1})\|}\leq 2\left(\frac{\sqrt{\kappa_i}-1}{\sqrt{\kappa_i}+1}\right)^{t_i}$}{
				$t_i = \lfloor et_i\rfloor$\;
				\If{$t+t_i>T$}{
					break\;
				}
				$x_{t-1}= x_t$\;
				$x=$Anderson Acceleration($x_t,t_i,\mu_i, L_i$) //can be replaced by other algorithms\;
				$t = t+t_i, x_t = x$\;
			}
			\If{$\|\nabla f(x_t)\| > \|\nabla f(x_{t-1})\|$}{
				$x_t = x_{t-1}$\;
			}
		}
	}
	\Return $x_t$
\end{algorithm}

The convergence result of our Algorithm \ref{alg:guess} is stated in the following Theorem \ref{thm:guess}.
The detailed proof is deferred to Appendix \ref{app:pfthm3}.
Note that we only prove the quadratic case for Theorem \ref{thm:guess}, but it is possible to extend it to the general case.

\begin{theorem}
	\label{thm:guess}
	Without knowing the parameters $\mu$ and $L$, Algorithm \ref{alg:guess} achieves $O(\sqrt{\kappa}\ln\frac{1}{\epsilon}+\sqrt{\kappa}(\ln\kappa\ln B)^2)$ convergence rate for obtaining an $\epsilon$-approximate solution of problem (\ref{prob:quad}), where $\kappa=L/\mu$, and $B$ can be any number as long as the eigenvalue spectrum belongs to $[\delta,B\delta]$. 
\end{theorem}
\noindent\textbf{Remark:} We provide a simple example to show why this guessing algorithm is useful.
Note that algorithms usually need the (exact) parameters $\mu$ and $L$ to set the step size.
Without knowing the exact values $\mu$ and $L$, one needs to approximate these parameters once at the beginning.
Let $\mu'=\frac{1}{c_1}\mu$ and $L'=c_2L$ denote the approximated values, where $c_1, c_2\geq 1$.
Without guessing them dynamically, one fixes $\mu'$ and $L'$ all the time in its algorithm. 
According to the lower bound $\Omega(\sqrt{\kappa}\ln\frac{1}{\epsilon})$, we know that its convergence rate cannot be better than $O(\sqrt{\kappa'}\ln \frac{1}{\epsilon})=O(\sqrt{c_1c_2\kappa}\ln \frac{1}{\epsilon})$, where $\kappa'=L'/\mu'$.
However, if one combines with our Algorithm \ref{alg:guess} (guessing the parameters dynamically), 
the convergence rate can be improved to $O(\sqrt{\kappa}\ln\frac{1}{\epsilon}+\sqrt{\kappa}(\ln\kappa\ln (c_1c_2\kappa))^2)$ according to our Theorem \ref{thm:guess} by letting $\delta=\mu'$ and $B\delta=L'$ (hence $B=c_1c_2\kappa$).
Note that there is no $\epsilon$ (accuracy) in the second term $\sqrt{\kappa}(\ln\kappa\ln (c_1c_2\kappa))^2$. Thus the rate turns to the optimal $O(\sqrt{\kappa}\ln\frac{1}{\epsilon})$ when $\epsilon\rightarrow 0$.
To achieve an $\epsilon$-approximate solution, our guessing algorithm can improve the convergence a lot especially for an imprecise estimate at the beginning (i.e., $c_1$ and $c_2$ are very large). 
The corresponding experimental results in Section \ref{app:expga} (see Figure~\ref{fig:7}) indeed validate our theoretical results.

\section{Experiments}
\label{sec:exp}
In this section, we conduct the numerical experiments on the real-world UCI datasets\footnote{The UCI datasets can be downloaded from \url{https://archive.ics.uci.edu/ml}} and synthetic datasets.
We compare the performance among these five algorithms: Anderson Acceleration (AA), Anderson-Chebyshev acceleration (AA-Cheby), vanilla Gradient Descent (GD), Nesterov's Accelerated Gradient Descent (NAGD) \citep{nesterov2014introductory} and Regularized Minimal Polynomial Extrapolation (RMPE) with $k=5$ (same as \citep{scieur2016regularized}).

Regarding the hyperparameters, we directly set them from their corresponding theoretical results. See Proposition 1 of \citep{lessard2016analysis} for GD and NAGD. For RMPE5, we follow the same setting as in \citep{scieur2016regularized}. For our AA/AA-Cheby, we set them according to our Theorem \ref{thm:opt} and \ref{thm:gel}.

Figure~\ref{fig:lreg} demonstrates the convergence performance of these algorithms in general nonlinear case and Figures~\ref{fig:1}--\ref{fig:5} demonstrate the convergence performance in quadratic case.
The last Figure~\ref{fig:7} demonstrates the convergence performance of these algorithms combined with our guessing algorithm (Algorithm \ref{alg:guess}).
The values of $m$ in the caption of figures denote the mixing parameter of Anderson acceleration algorithms (see Line 5 of Algorithm \ref{alg:am}).

\begin{figure}[!htb]
	\centering
	\begin{minipage}[htb]{0.25\textwidth}
		\includegraphics[width=\textwidth]{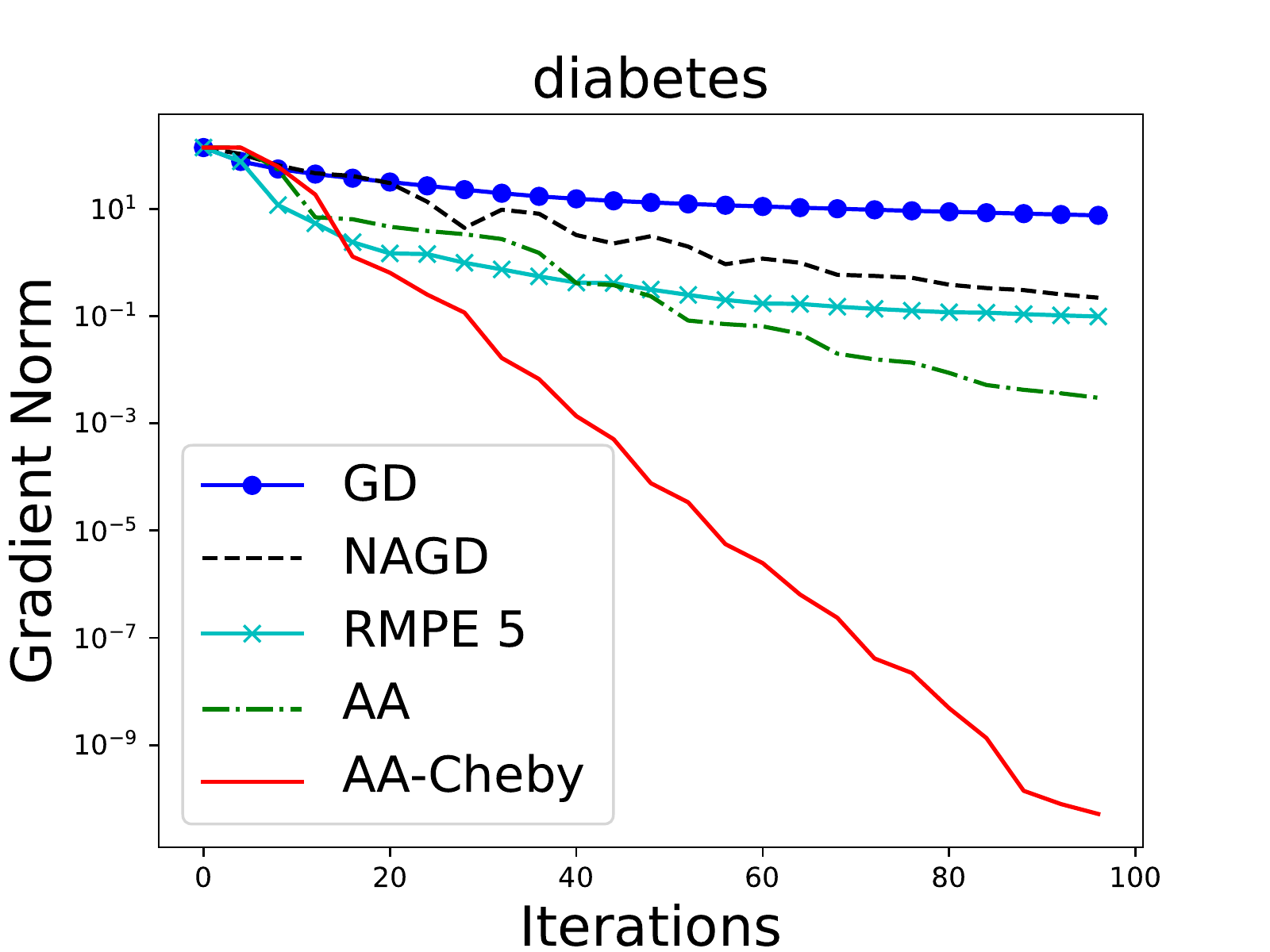}
	\end{minipage}%
	\begin{minipage}[htb]{0.25\textwidth}
		\includegraphics[width=\textwidth]{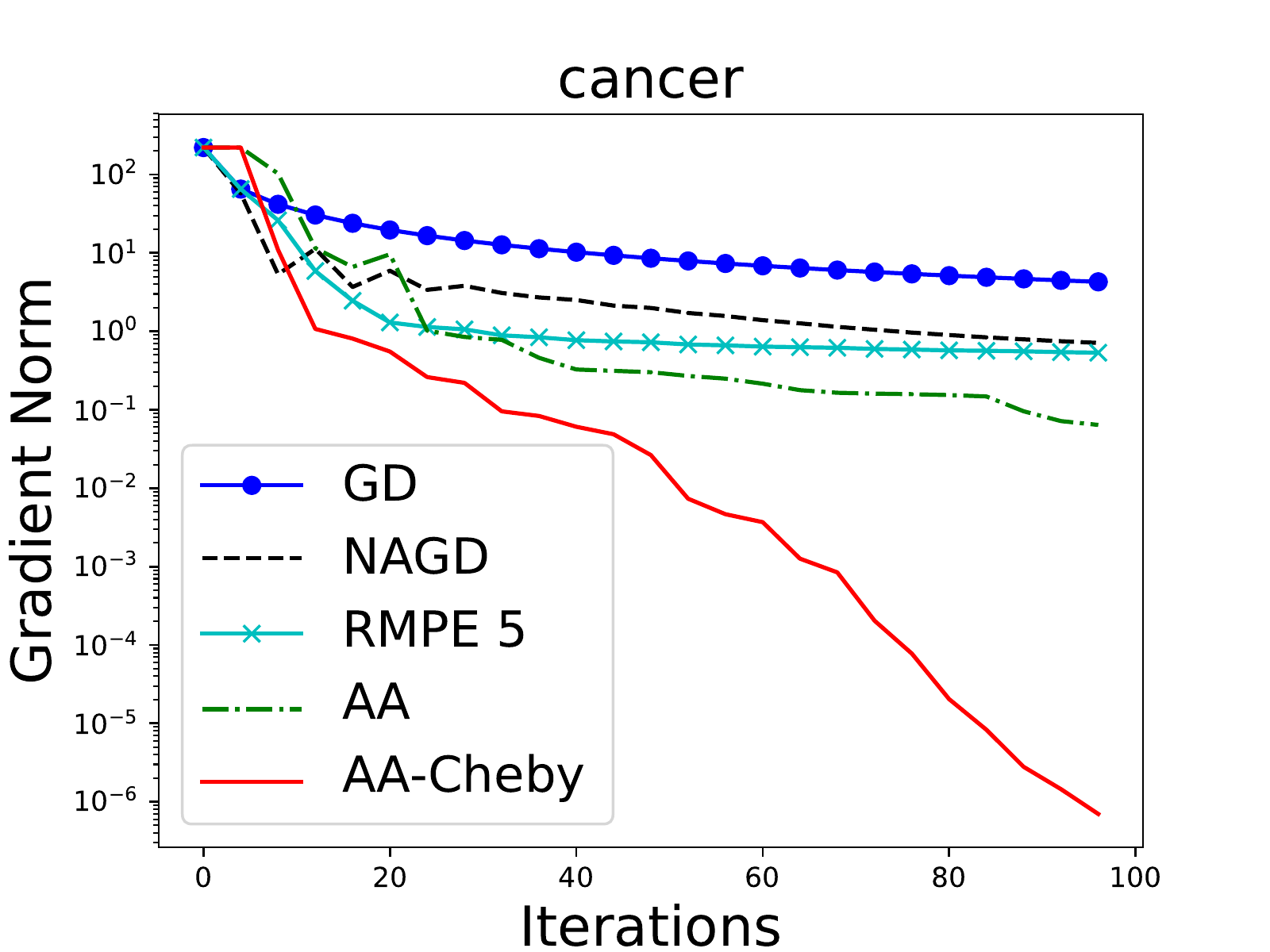}
	\end{minipage}
	\caption{Logistic regression, $m=3$}
	\label{fig:lreg}
\end{figure}

In Figure~\ref{fig:lreg}, we use the negative log-likelihood as the loss function $f$ (logistic regression), i.e., $f(\theta)= - \sum_{i=1}^n(y_i\log\phi(\theta^Tx_i) + (1-y_i)\log(1-\phi(\theta^Tx_i)))$, where $\phi(z)= 1/(1+\exp(-z))$. We run these five algorithms on real-world \emph{diabetes} and \emph{cancer} datasets which are standard UCI datasets.
The x-axis and y-axis represent the number of iterations and the norm of the gradient of loss function respectively.

\begin{figure}[!htb]
	\centering
	\begin{minipage}[htb]{0.25\textwidth}
		\includegraphics[width=\textwidth]{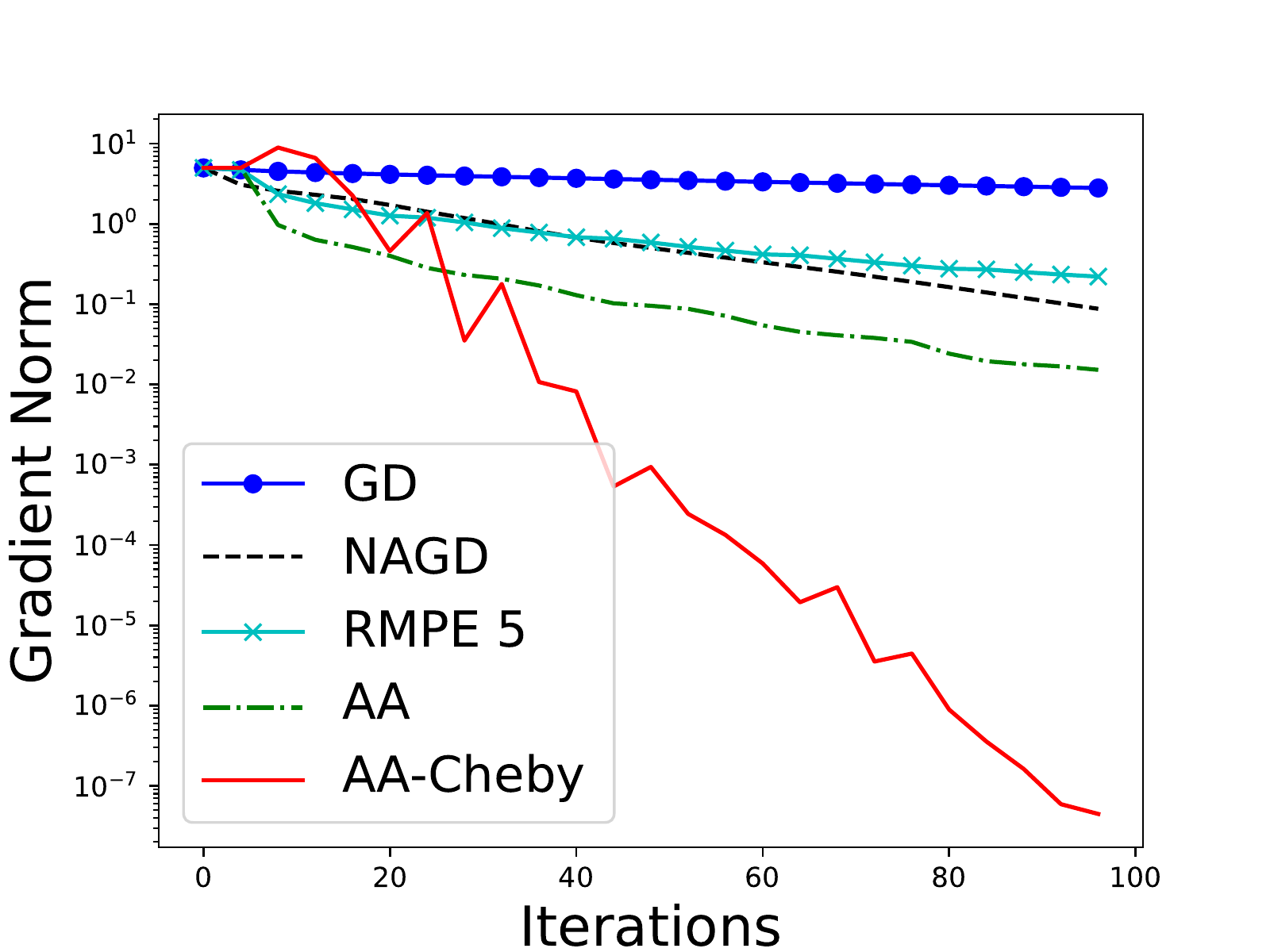}
	\end{minipage}%
	\begin{minipage}[htb]{0.25\textwidth}
		\includegraphics[width=\textwidth]{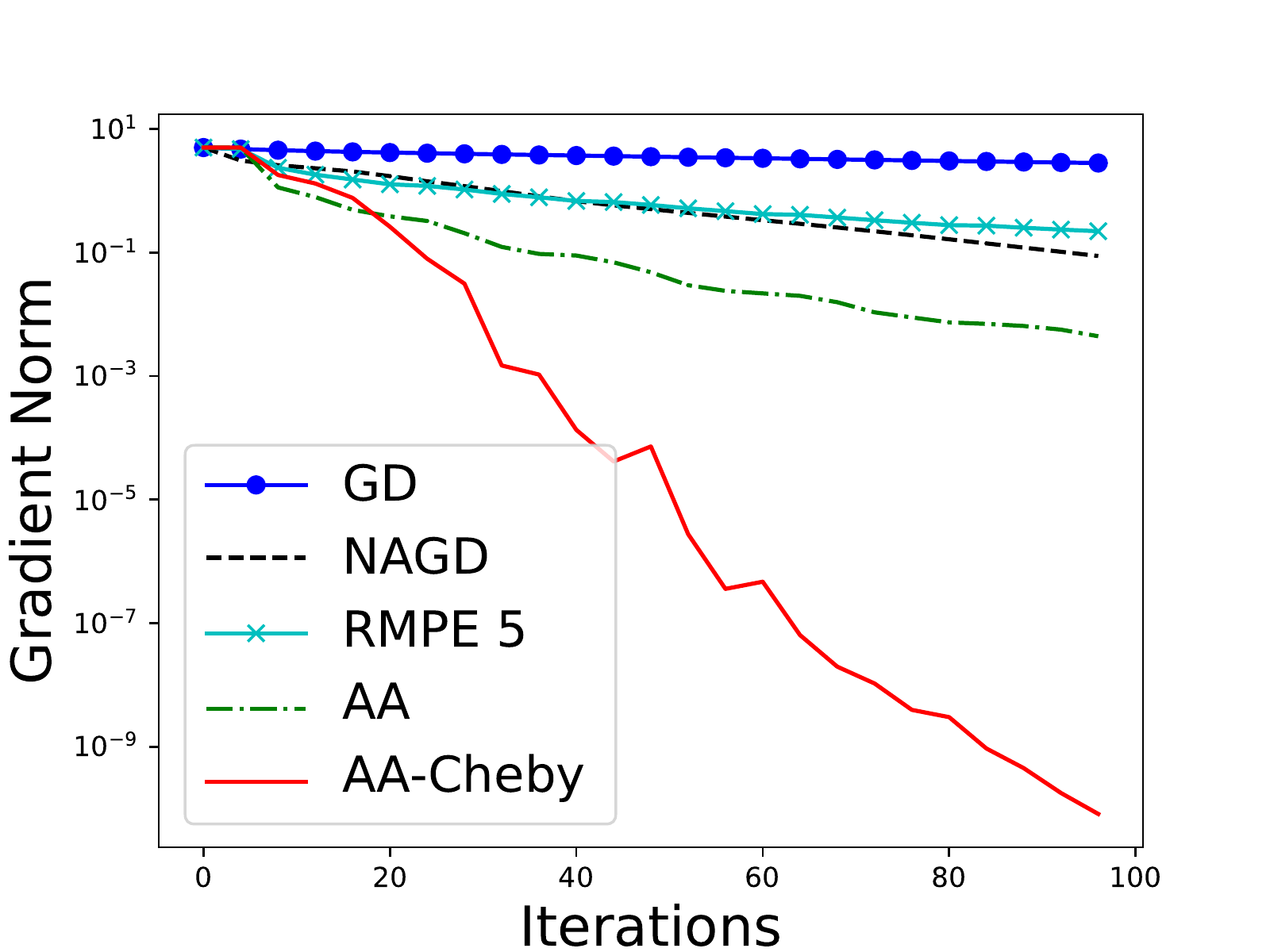}
	\end{minipage}
	\caption{$\kappa\in[0,500]$; $m=3$ (left), $m=5$ (right)}
	\label{fig:1}\vspace{-3mm}
\end{figure}

\begin{figure}[!htb]
	\centering
	\begin{minipage}[htb]{0.25\textwidth}
		\includegraphics[width=\textwidth]{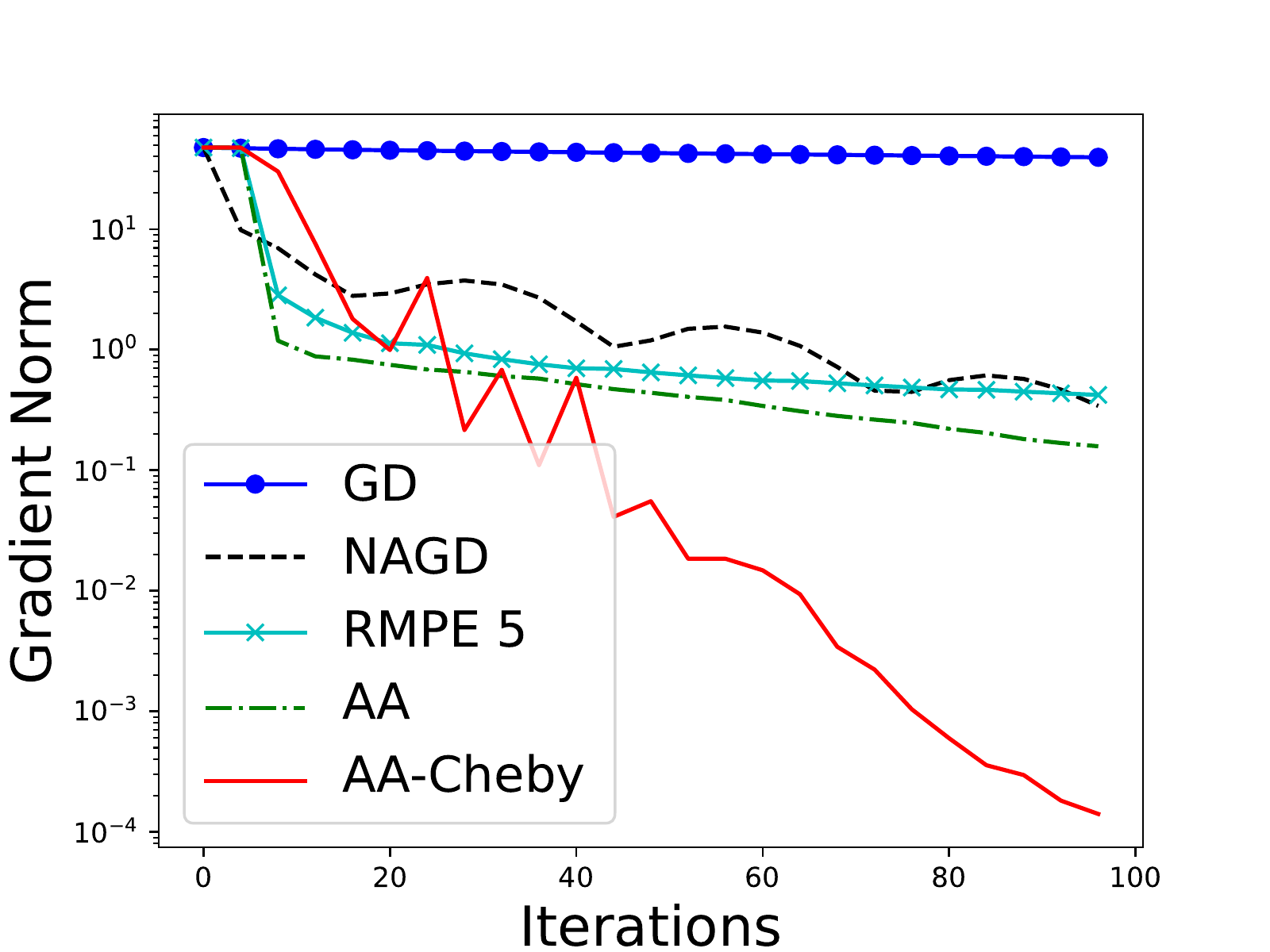}
	\end{minipage}%
	\begin{minipage}[htb]{0.25\textwidth}
		\includegraphics[width=\textwidth]{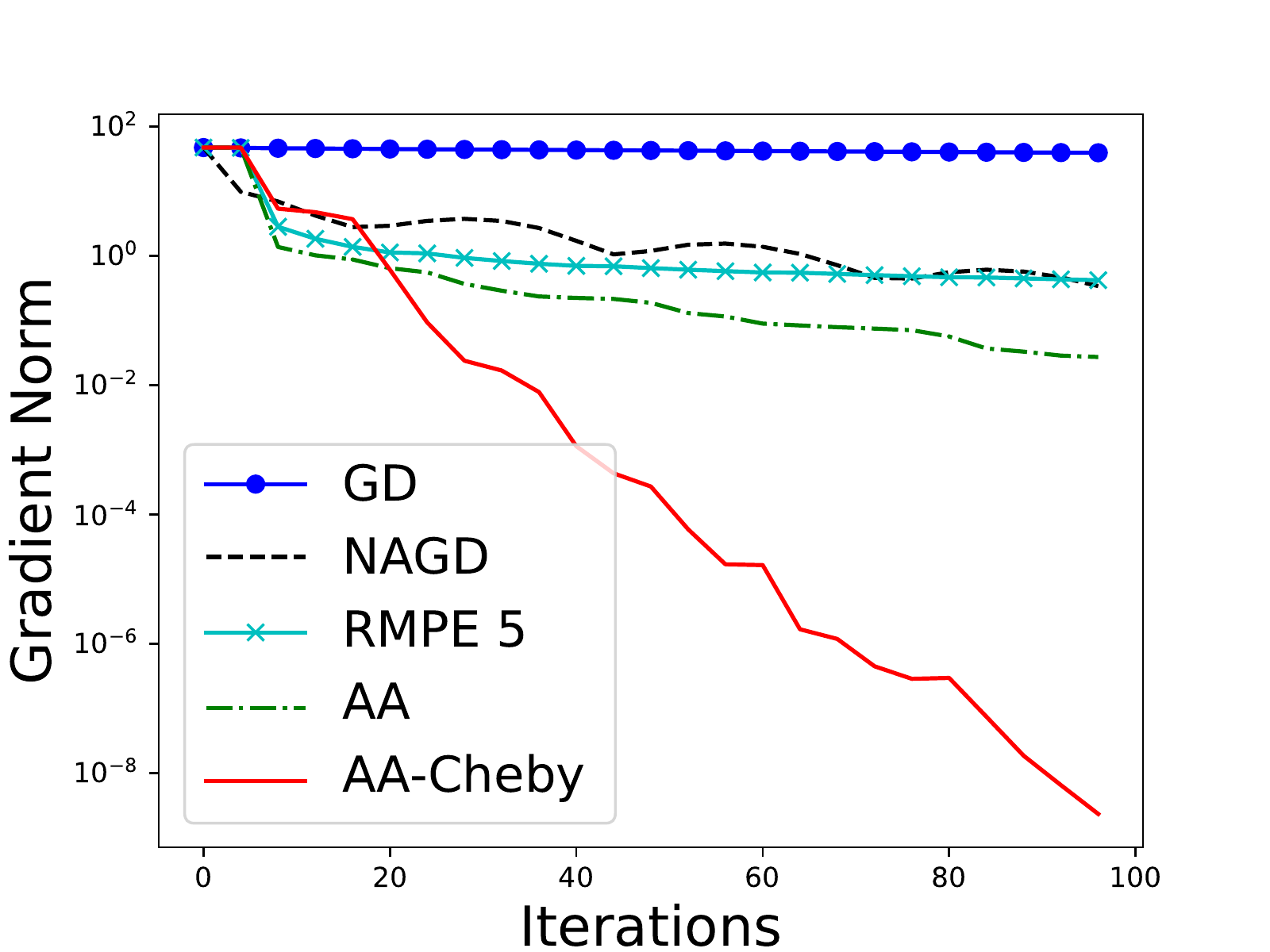}
	\end{minipage}\\
	\caption{$\kappa\in[500,2000]$; $m=3$ (left), $m=5$ (right)}
	\label{fig:3}\vspace{-3mm}
\end{figure}

\begin{figure}[!htb]
	\centering
	\begin{minipage}[htb]{0.25\textwidth}
		\includegraphics[width=\textwidth]{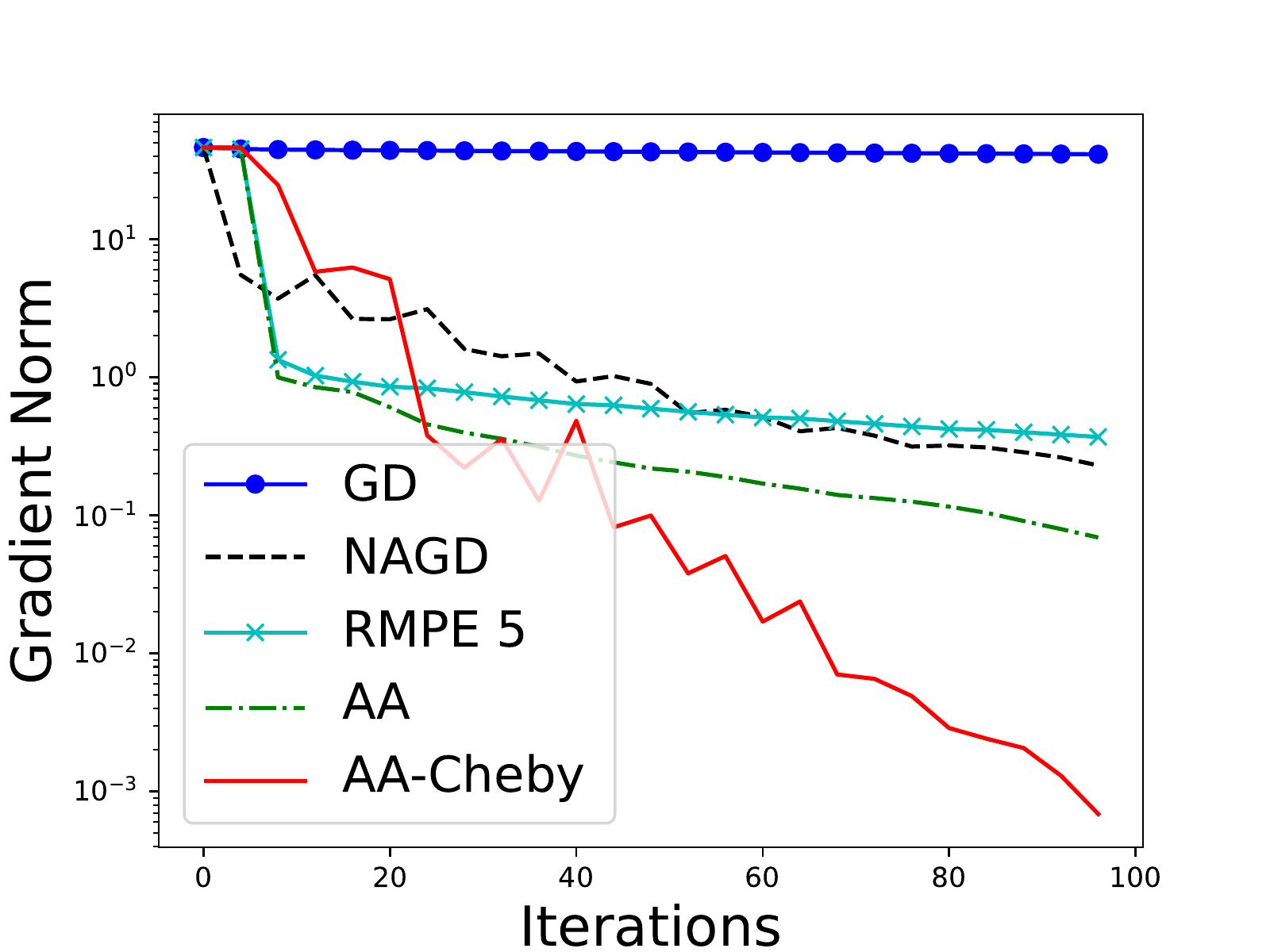}
	\end{minipage}%
	\begin{minipage}[htb]{0.25\textwidth}
		\includegraphics[width=\textwidth]{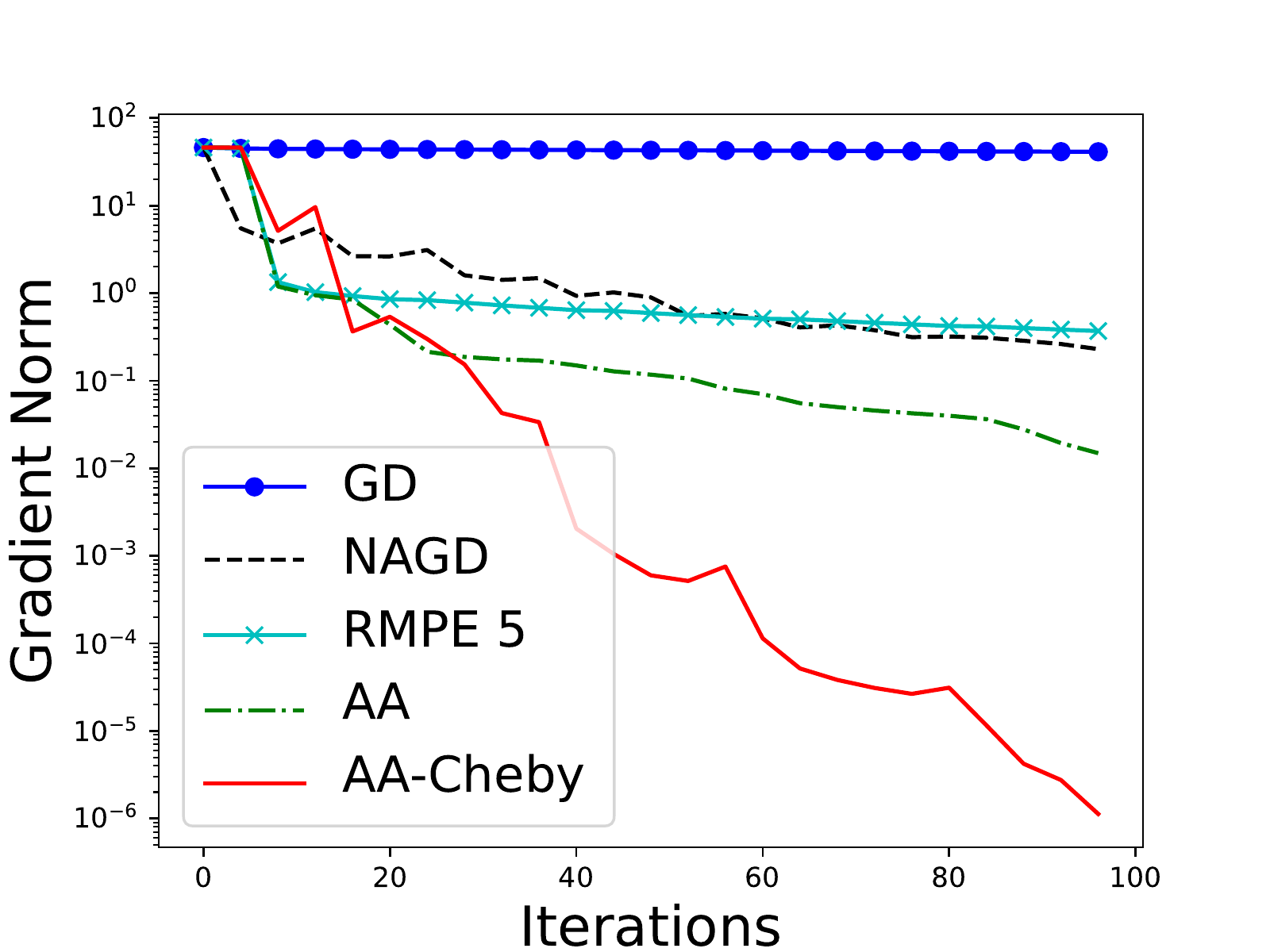}
	\end{minipage}\\
	\caption{$\kappa\in[2000,5000]$; $m=3$ (left), $m=5$ (right)}
	\label{fig:5}\vspace{-2mm}
\end{figure}

Figures~\ref{fig:1}--\ref{fig:5} demonstrate the convergence performance for the quadratic case, where $f(x)=\frac{1}{2}x^TAx-b^Tx$.
Concretely, we compared the convergence performance among these algorithms when the condition number $\kappa(A)$ and the mixing parameter $m$ are varied, e.g., the left figure in Figure \ref{fig:1} is the case $\kappa\in[0,500]$ and $m =3$. Recall that $m$ is the mixing parameter for Anderson acceleration algorithms (see Line 5 of Algorithm \ref{alg:am}).
We run these five algorithms on the synthetic datasets in which we randomly generate the $A$ and $b$ for the loss function $f$. Note that for randomly generated $A$ satisfying
the property of $A\in \mathcal{S}_{++}^d$, we randomly generate $B$ instead and let $A \triangleq B^TB$.

In conclusion, Anderson acceleration methods converge the fastest no matter it is a quadratic function or general function in all of our experiments.
The efficient Anderson acceleration methods can be viewed as the extension of momentum methods (e.g., NAGD) 
since GD is the special case of Anderson Acceleration with $m=0$, and to some extent NAGD can be viewed as $m=1$.
Combined with our theoretical results (i.e., optimal convergence rate in quadratic case and linear-quadratic convergence in general case), the experimental results validate that Anderson acceleration methods are efficient both in theory and practice.

\subsection{Experiments for Guessing Algorithm}
\label{app:expga}
In this section, we conduct the experiments for guessing the hyperparameters (i.e., $\mu, L$) dynamically using our Algorithm \ref{alg:guess}.
\begin{figure}[!htb]
	\centering
	\begin{minipage}[h]{0.25\textwidth}
		\includegraphics[width=\textwidth]{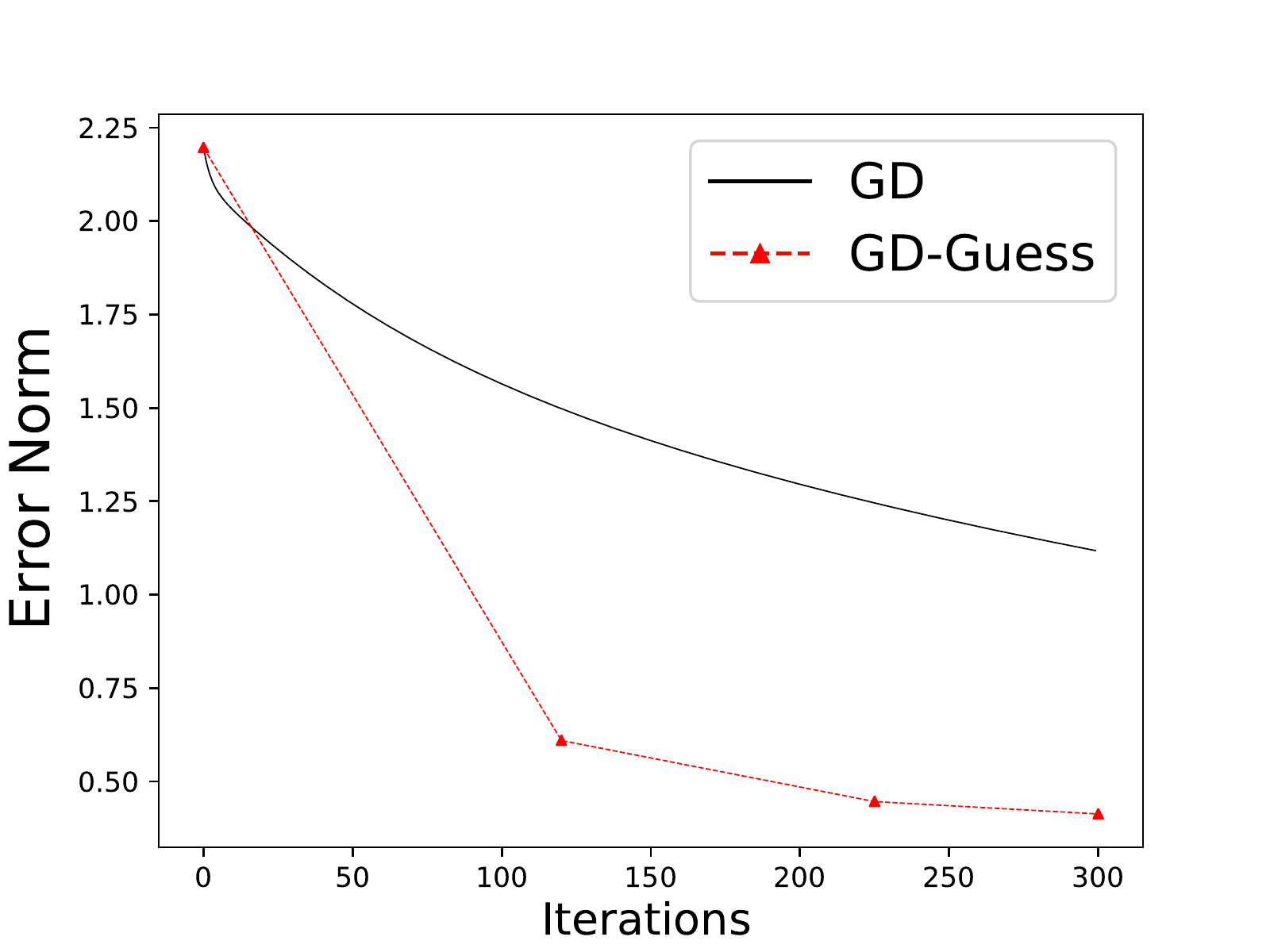}
		{ (a) Gradient Descent}
	\end{minipage}%
	\begin{minipage}[h]{0.25\textwidth}
		\includegraphics[width=\textwidth]{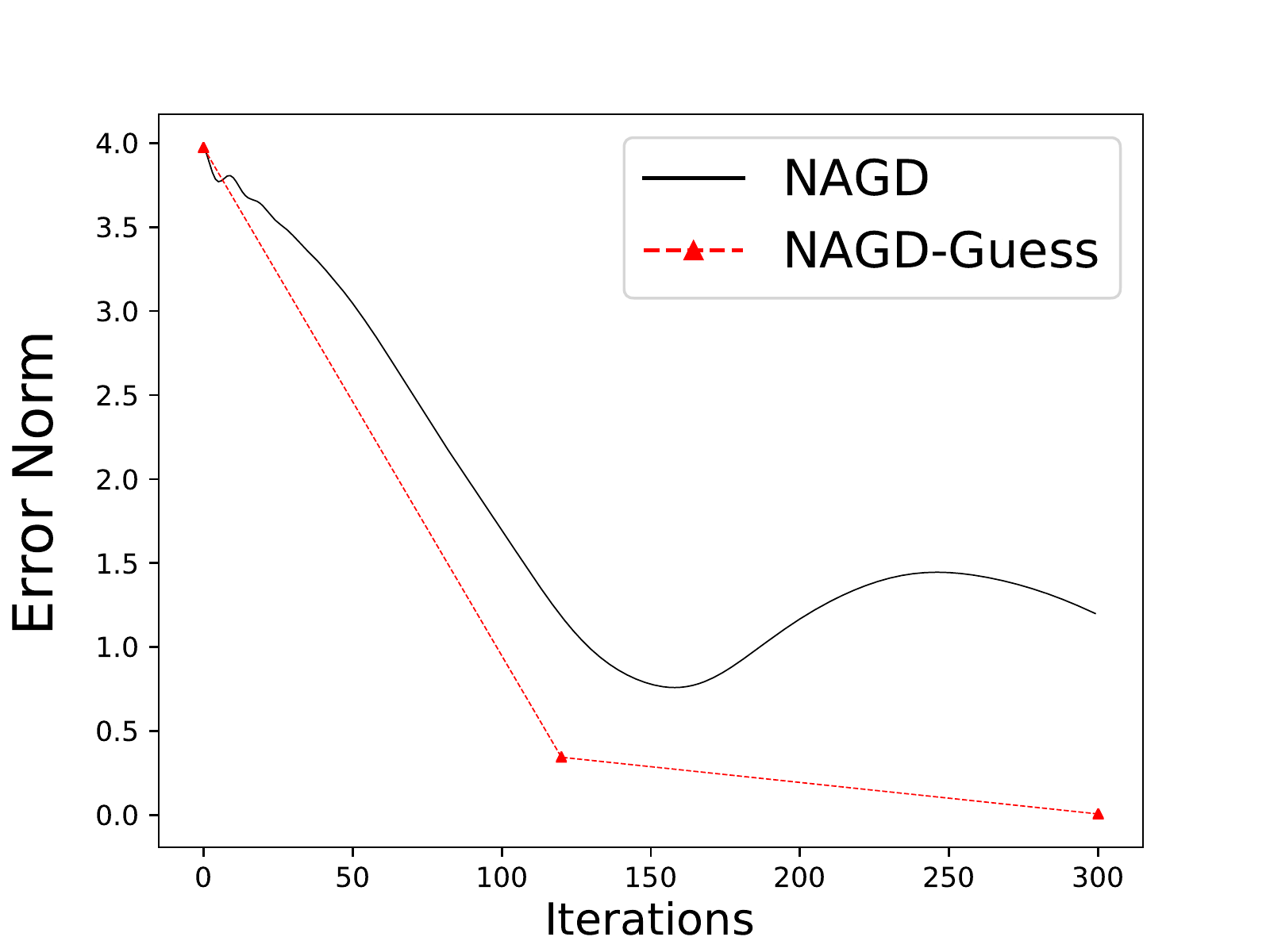}
		{\quad (b) Nesterov's AGD}
	\end{minipage}\\
	\begin{minipage}[h]{0.25\textwidth}
		\includegraphics[width=\textwidth]{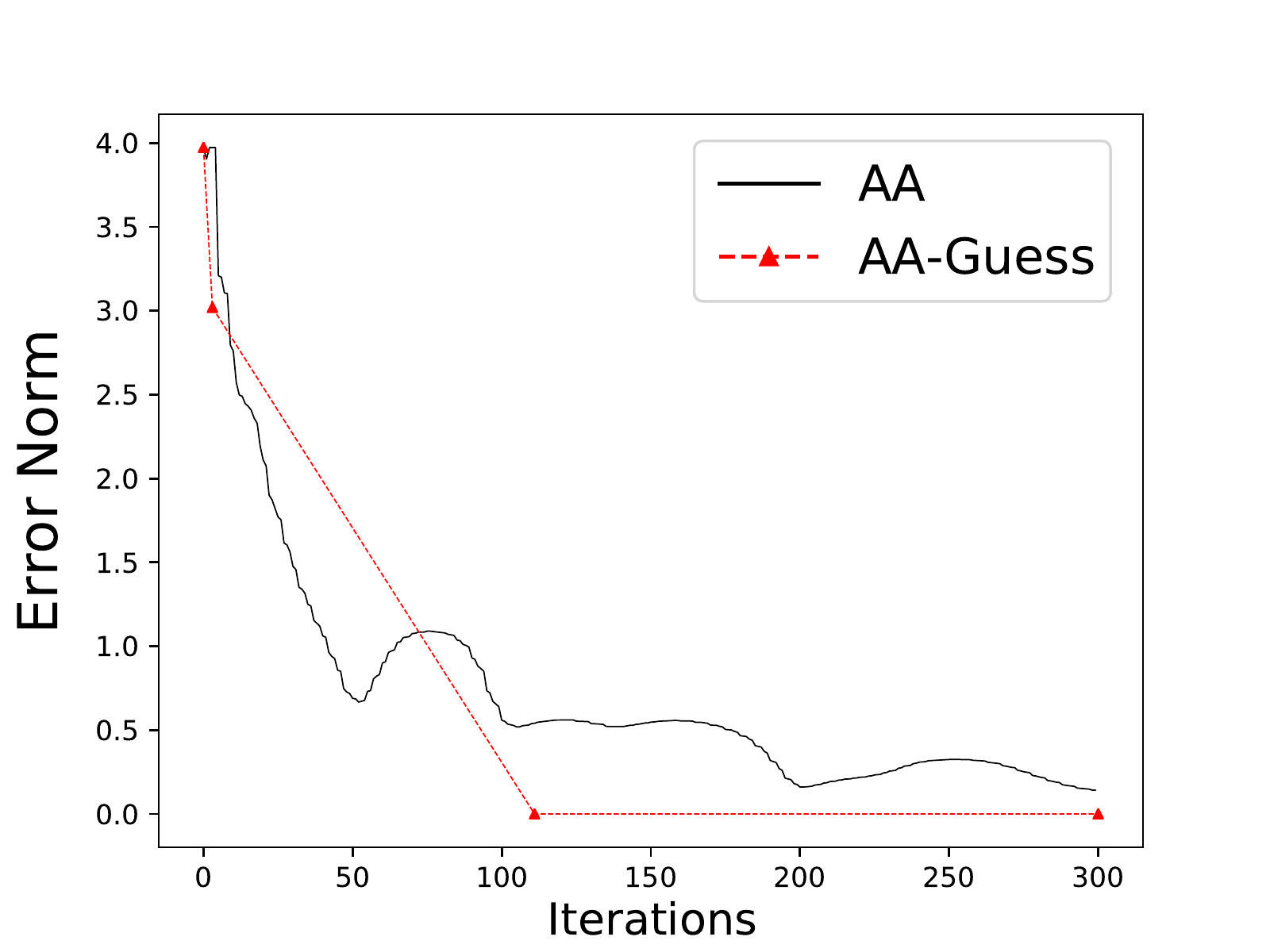}
		{\quad (c) Anderson Acceleration}
	\end{minipage}%
	\begin{minipage}[h]{0.25\textwidth}
		\includegraphics[width=\textwidth]{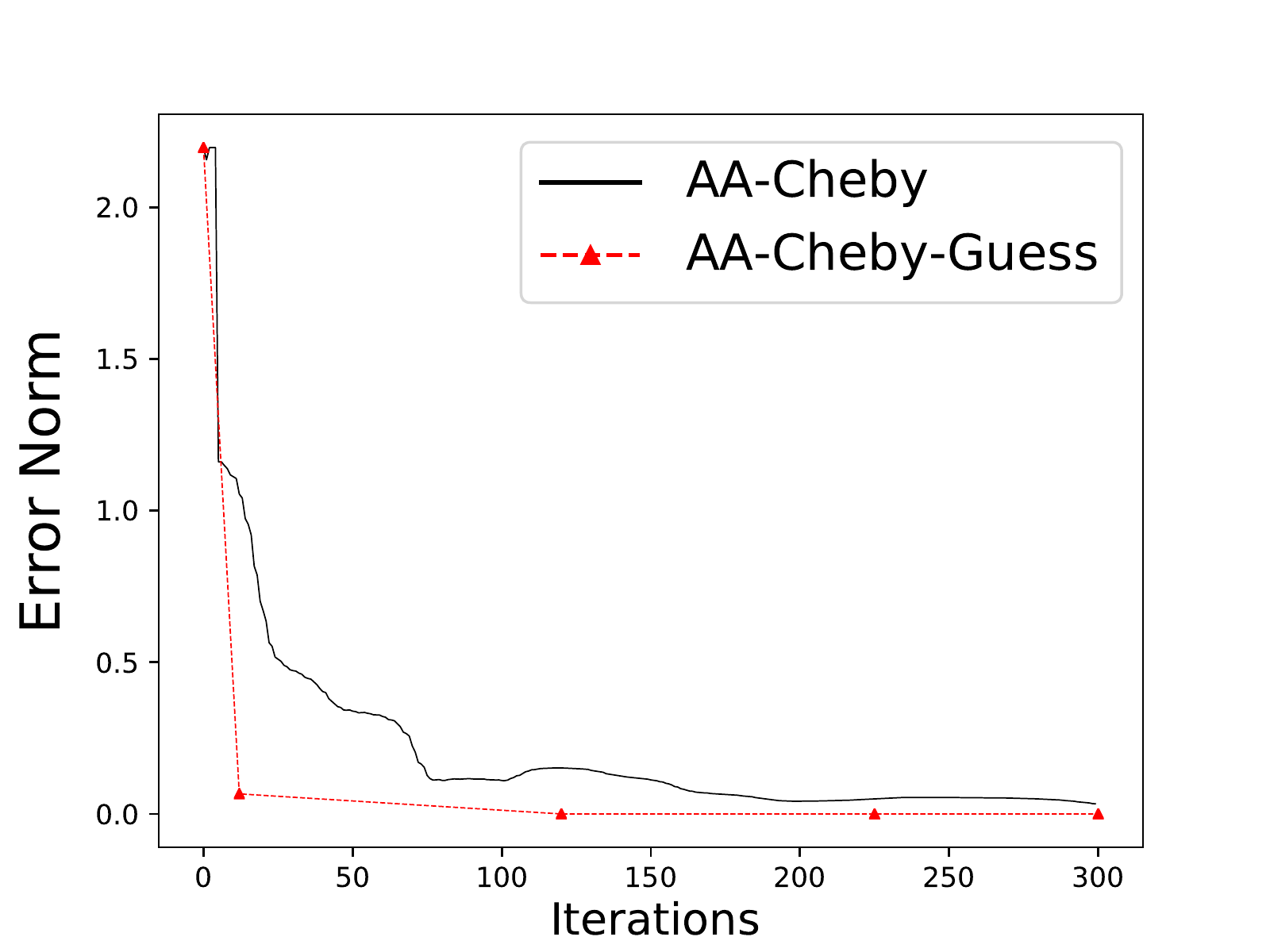}
		{\quad (d) Anderson-Chebyshev}
	\end{minipage}
	\caption{Algorithms with/without guessing algorithm}
	\label{fig:7}
\end{figure}

In Figure \ref{fig:7}, we separately consider these algorithms.
For each of them, we compare its convergence performance between its original version and the one combined with our guessing algorithm (Algorithm \ref{alg:guess}).
The experimental results show that all these four algorithms combined with our guessing algorithm achieve much better performance than their original versions.
Thus it validates our theoretical results (see Theorem \ref{thm:guess} and its following Remark).

\section{Conclusion}

In this paper, we prove that Anderson acceleration with Chebyshev polynomial can achieve the optimal convergence rate $O(\sqrt{\kappa}\ln\frac{1}{\epsilon})$, which improves the previous result $O(\kappa\ln\frac{1}{\epsilon})$ provided by \citep{toth2015convergence}.
Thus it can deal with ill-conditioned problems (condition number $\kappa$ is large) more efficiently.
Furthermore, we also prove the linear-quadratic convergence of Anderson acceleration for minimizing general nonlinear problems.
Besides, if the hyperparameters (e.g., the Lipschitz smooth parameter $L$) are not available, we propose a guessing algorithm for guessing them dynamically and also prove a similar convergence rate.
Finally, the experimental results demonstrate that the efficient Anderson acceleration methods converge significantly faster than other algorithms. 
This validates that Anderson-Chebyshev acceleration is efficient both in theory and practice.

\newpage
\subsubsection*{Acknowledgements}
Zhize was supported by the Office of Sponsored Research of KAUST, through the Baseline Research Fund of Prof. Peter Richt{\'a}rik.
Jian was supported in part by the National Natural Science Foundation of China Grant 61822203, 61772297, 61632016, 61761146003, and the Zhongguancun Haihua Institute for Frontier Information Technology and Turing AI Institute of Nanjing.
The authors also would like to thank Francis Bach, Claude Brezinski, Rong Ge, Damien Scieur, Le Zhang and anonymous reviewers for useful discussions and suggestions.

\bibliographystyle{plainnat}
\bibliography{ac}

\onecolumn
\newpage
\appendix

\section{GMRES vs. Anderson Acceleration($m =\infty$)}
\label{app:gam}

In this appendix, in order to better understand the efficient Anderson acceleration method, we review the equivalence between the well-known Krylov subspace method GMRES \citep{saad1986gmres} and Anderson acceleration without truncation (i.e., $m =\infty$ or large enough in Line 5 of Algorithm \ref{alg:am}) in linear case.
We emphasize that in this paper we focus on the more general hard cases where $m$ is small (since $m$ usually is finite and not very large in practice) and also general nonlinear case.

Consider the problem of solving the linear system $Ax=b$, with a nonsingular matrix $A$. This is equivalent to solving the fixed point $x=G(x)=x-\nabla f(x)$, where $\nabla f(x)=Ax-b$. Let $r_i$ denote the residual in the point $x_i$, i.e., $r_i=b-Ax_i$.
The GMRES method is an effective iterative method for linear system which has the property of minimizing the norm of the residual vector over a Krylov subspace at every step.
\begin{align}\label{eq:tgmres}
x_t^{\Gm} = \arg\min\{\|b-Ax\|_2: x=x_0+y, y\in \mathcal{K}_t\}
\end{align}
Note that the Krylov space
$\mathcal{K}_t$ is the linear span of the first $t$ gradients
and $\mathcal{K}_n$ can span the whole space $\mathbb{R}^n$. Hence the method arrives the exact solution after $n$ iteration. It is also theoretically equivalent to the Generalized Conjugate Residual method (GCR).

Now we show that $x_{t+1}^{\Am}=G(x_t^{\Gm})$ to indicate the equivalence, under the assumption $0<\|r_i\|_2<\|r_{i-1}\|_2$ for $1\leq i\leq t$. $x_t^{\Gm}$ and $x_{t+1}^{\Am}$ denote the $t$-th GMRES iterative point and $t+1$-th Anderson Acceleration iterative point, respectively.
Let mixing parameters $\beta_t=1$ for all $t$.
Then, we deduce the $x_{t+1}^{\Am}$ as follows:
\begin{align}
x_{t+1}^{\Am}&=\sum_{i=0}^{t}{\alpha_i^t G(x_i)} \qquad \because m_t=t\\
&=\sum_{i=0}^{t}{\alpha_i^t x_i}+\sum_{i=0}^{t}{\alpha_i^t(G(x_i)-x_i)}\\
&=\sum_{i=0}^{t}{\alpha_i^t x_i}+ \sum_{i=0}^{t}{\alpha_i^t F_i} \label{eq:min}
\end{align}
Note that the second term in (\ref{eq:min}) is the same as we minimized in Line 7 of Algorithm \ref{alg:am}.
This step also can be transformed to an unconstrained version as follows:
\begin{equation}\label{eq:r0}
\min_{(\alpha_1^t,\ldots,\alpha_{t}^t)^T}\|F_0-\sum_{i=1}^{t}{\alpha_i^t (F_0-F_{i})}\|_2
\end{equation}
The $\alpha_0^t$ equals to $1-\sum_{i=1}^{t}{\alpha_i^t}$.
Note that $F_0-F_{i}=b-Ax_0-(b-Ax_i)=A(x_i-x_0)$ and $F_0=r_0=b-Ax_0$. Replacing these equations into (\ref{eq:r0}), we have

\begin{align}
&\min_{(\alpha_1^t,\ldots,\alpha_{t}^t)^T}\|F_0-\sum_{i=1}^{t}{\alpha_i^t (F_0-F_{i})}\|_2  \label{eq:tam1}\\
=&\min_{(\alpha_1^t,\ldots,\alpha_{t}^t)^T}\|b-Ax_0-\sum_{i=1}^{t}{\alpha_i^t A(x_i-x_0)}\|_2\\
=&\min_{(\alpha_1^t,\ldots,\alpha_{t}^t)^T}\|b-A\Big(x_0+\sum_{i=1}^{t}{\alpha_i^t(x_i-x_0)}\Big)\|_2 \label{eq:tam}
\end{align}

Comparing (\ref{eq:tam}) with (\ref{eq:tgmres}), if $\{y_i=(x_i-x_0): 1\leq i\leq t\}$ form a basis of Krylov subspace $\mathcal{K}_t$, then we have the following equations (easily from (\ref{eq:tam1})-(\ref{eq:tgmres})). Note that the Krylov subspaces $\mathcal{K}$ are defined by $(r_0,A)$, i.e., $\mathcal{K}_i=\mathrm{span}\{r_0,Ar_0,\ldots,A^{i-1}r_0\}$.
\begin{align}
x_t^{\Gm}&=x_0+\sum_{i=1}^{t}{\alpha_i^t(x_i-x_0)}\\ r_t^{\Gm}&=b-Ax_t^{\Gm}=F_0-\sum_{i=1}^{t}{\alpha_i^t (F_0-F_{i})} \label{eq:rt}
\end{align}
Now we continue to deduce the $x_{t+1}^{\Am}$ from (\ref{eq:min}) to finish the proof of equivalence.
\begin{align}
x_{t+1}^{\Am} &=\sum_{i=0}^{t}{\alpha_i^t x_i}+ \sum_{i=0}^{t}{\alpha_i^t F_i}\\
&=x_0+\sum_{i=1}^{t}{\alpha_i^t (x_i-x_0)}+ F_0-\sum_{i=1}^{t}{\alpha_i^t (F_0-F_{i})}
\label{eq:t+1}\\
&=x_t^{\Gm}+b-Ax_t^{\Gm}\\
&=x_t^{\Gm}-\nabla f(x_t^{\Gm})\\
&=G(x_t^{\Gm})
\end{align}
Now the only remaining thing is to show that $\{y_i=(x_i-x_0): 1\leq i\leq t\}$ form the basis of Krylov subspace $\mathcal{K}_t$. This can be proved by induction. For $t=1$, $y_1=x_1-x_0=G(x_0)-x_0=x_0-(Ax_0-b)-x_0=r_0$. Now, assuming that $\{y_i=(x_i-x_0): 1\leq i\leq t\}$ form the basis of $\mathcal{K}_t$, we show that
\begin{align}
y_{t+1}&=x_{t+1}-x_0 \notag\\
&=\sum_{i=1}^{t}{\alpha_i^t (x_i-x_0)}+ F_0-\sum_{i=1}^{t}{\alpha_i^t (F_0-F_{i})} \label{eq:abcd}\\
&=\sum_{i=1}^{t}{\alpha_i^t y_i}+r_t^{\Gm}, \label{eq:lastk}
\end{align}
where (\ref{eq:abcd}) follows from (\ref{eq:t+1}), and (\ref{eq:lastk}) follows from (\ref{eq:rt}).
The first term in (\ref{eq:lastk}) belongs to $\mathcal{K}_t$ by induction. The second term $r_t^{\Gm}\in\mathcal{K}_{t+1}$ (From~(\ref{eq:tgmres})) and $r_t^{\Gm}\not\in\mathcal{K}_t$ since the assumption $0<\|r_i\|_2<\|r_{i-1}\|_2$ for $1\leq i\leq t$. Hence $y_{t+1}\in\mathcal{K}_{t+1}$.

\section{Missing Proofs}\label{app:pf}
In this appendix, we provide the proof details for Theorem \ref{thm:gel} (Appendix \ref{app:pfthm1}), Lemma \ref{lm:boundft1} (Appendix \ref{app:pflm2}) and Theorem \ref{thm:guess} (Appendix \ref{app:pfthm3}).

\subsection{Proof of Theorem \ref{thm:gel}}
\label{app:pfthm1}
For the iteration $t+1$, we have $F_t=F(x_t)=-\frac{2}{L+\mu}\nabla f(x_t)$ according to $\lambda =\frac{2}{L+\mu}$, where $\mu$ and $L$ are defined in Assumption \ref{asp:1}.
First, we recall the form of $F_{t+1}$ (i.e. (\ref{eq:sf2})) as
\begin{align}
F_{t+1}
&= G_{t+1} - \sum_{i=0}^{m_t}{\alpha_i^t G_{t-i}} + (1-\beta_t)\mathcal{F}, \label{eq:f2}
\end{align}
and the definition of  $\mathcal{F}$ as
\begin{equation}\label{eq:deff}
\mathcal{F} \triangleq \sum_{i=0}^{m_t}{\alpha_i^tF_{t-i}}.
\end{equation}
Now we bound the first two terms in RHS of (\ref{eq:f2}) by combining the first and third term of (\ref{eq:sint}) as follows:
\begin{align}
&G_{t+1} - \sum_{i=0}^{m_t}{\alpha_i^t G_{t-i}}  \notag\\
&= \sum_{i=1}^{m_t}{\alpha_i^t}\int_0^1\biggl(G'\Bigl(x_t+u(x_{t+1}-x_t)\Bigr)
- G'\Bigl(x_t+u(x_{t-i}-x_t)\Bigr)\biggr)(x_{t-i}-x_t)\,du   \notag\\
&\qquad + \int_0^1G'\Bigl(x_t+u(x_{t+1}-x_t)\Bigr)\beta_t\mathcal{F}\,du. \label{eq:int}
\end{align}
To bound the Equation (\ref{eq:int}), we recall that $G_t=G(x_t)=x_t+F_t$ and $F_t=-\frac{2}{L+\mu}\nabla f(x_t)$. Hence
\begin{equation}\label{eq:gp}
G_t' = I+F_t' = I - \frac{2}{L+\mu}\nabla^2f(x_t).
\end{equation}
Due to the Lipschitz continuity of Hessian $\nabla^2f$ (see (\ref{asp3})), we have
\begin{align}
\|G'(x) - G'(y)\| &= \frac{2}{L+\mu}\|\nabla^2f(x) - \nabla^2f(y)\| \notag \\
& \leq \frac{2\gamma}{L+\mu} \|x-y\|.
\end{align}
Now the first term in (\ref{eq:int}) can be bounded as follows:
\begin{align}
&\sum_{i=1}^{m_t}{\alpha_i^t}\int_0^1\biggl(G'\Bigl(x_t+u(x_{t+1}-x_t)\Bigr)
- G'\Bigl(x_t+u(x_{t-i}-x_t)\Bigr)\biggr)(x_{t-i}-x_t)\,du   \notag\\
&\leq \sum_{i=1}^{m_t}{\alpha_i^t}\frac{\gamma\|x_{t+1}-x_{t} -(x_{t-i}-x_t)\|\|x_{t-i}-x_t\|}{L+\mu}.
\label{eq:int1}
\end{align}
Using (\ref{eq:x2}) to replace $x_{t+1}$ and combining with (\ref{eq:deff}), we have
\begin{align}
&\|x_{t+1}-x_{t} -(x_{t-i}-x_t)\|  \notag\\
&= \|\sum_{i=1}^{m_t}{\alpha_i^t (x_{t-i}-x_{t})} +\beta_t\mathcal{F}  -(x_{t-i}-x_t)\| \notag\\
&\leq (\sqrt{m}\n{\alpha}+1)\max_{i\in[1,m]}{\|x_t-x_{t-i}\|} +\beta_t\|\mathcal{F}\| \label{eq:xx}\\
&= (\sqrt{m}\n{\alpha}+1)\Delta_t +\beta_t\|\mathcal{F}\|, \label{eq:int1b}
\end{align}
where (\ref{eq:xx}) uses triangle inequality and Cauchy–Schwarz inequality.
Now, pluging (\ref{eq:int1b}) into (\ref{eq:int1}), we get
\begin{align}
(\ref{eq:int1}) \leq \frac{\sqrt{m}\n{\alpha}\gamma((\sqrt{m}\n{\alpha}+1)\Delta_t +\beta_t\|\mathcal{F}\|)\Delta_t}{L+\mu}, \label{eq:int3}
\end{align}
where (\ref{eq:int3}) uses Cauchy–Schwarz inequality.
Then, we bound the second term in (\ref{eq:int}) as follows:
\begin{align}
&\int_0^1G'\Bigl(x_t+u(x_{t+1}-x_t)\Bigr)\beta_t\mathcal{F}\,du \notag\\
&=\int_0^1\biggl(I-\frac{2}{L+\mu}\nabla^2f\Bigl(x_t+u(x_{t+1}-x_t)\Bigr)\biggr)\beta_t\mathcal{F}\,du \notag \\
& \leq \Bigl(1-\frac{2\mu}{L+\mu}\Bigr)\beta_t\|\mathcal{F}\|. \label{eq:int2}
\end{align}
Now, we recall $F_{t+1}$ here:
\begin{align*}
F_{t+1} = G_{t+1} - \sum_{i=0}^{m_t}{\alpha_i^t G_{t-i}} + (1-\beta_t)\mathcal{F}
\quad\text{same as}~(\ref{eq:f2})
\end{align*}
Then, according to (\ref{eq:int}), (\ref{eq:int1}), (\ref{eq:int3}) and (\ref{eq:int2}), we have
\begin{align*}
\|F_{t+1}\| &\leq  \frac{\sqrt{m}\n{\alpha}\gamma((\sqrt{m}\n{\alpha}+1)\Delta_t +\beta_t\|\mathcal{F}\|)\Delta_t}{L+\mu}
+ \Bigl(1-\frac{2\mu}{L+\mu}\Bigr)\beta_t\|\mathcal{F}\|
+(1-\beta_t)\|\mathcal{F}\| \\
&= \frac{\gamma(m\ns{\alpha}+\sqrt{m}\n{\alpha})\Delta_t^2}{L+\mu} + \frac{\gamma\sqrt{m}\n{\alpha}\beta_t\Delta_t\|\mathcal{F}\|}{L+\mu}
+\Bigl(1-\frac{2\mu}{L+\mu}\beta_t\Bigr)\|\mathcal{F}\|.
\end{align*}
Recall that $F_t=-\frac{2}{L+\mu}\nabla f(x_t)$. According to (\ref{eq:deff}) and (\ref{eq:transt}), we have $\|\mathcal{F}\| \leq  \|F_t\|$. Thus, we have
\begin{align}
\|\nabla f(x_{t+1})\| &\leq \frac{\gamma (m\ns{\alpha}+\sqrt{m}\n{\alpha}) \Delta_t^2}{2}
+ \frac{\gamma\sqrt{m}\n{\alpha}\beta_t\Delta_t\|\nabla f(x_t)\|}{L+\mu} \notag\\
&\qquad
+ \Bigl(1-\frac{2\mu}{L+\mu}\beta_t\Bigr)\|\nabla f(x_t)\|. \label{eq:finalxx}
\end{align}
Now, we bound $\n{\alpha}\leq\frac{2\tilde{\kappa}}{L+\mu}$ to finish the proof for Theorem \ref{thm:gel}.
First we recall that the $\alpha$ satisfies problem (\ref{eq:transt}), i.e., $\alpha =\arg\min_{\alpha} \|F_t- \mathcal{B}\alpha\|_2$. We use the QR decomposition for $\mathcal{B}$ and let $\mathcal{B}=QR$, where $Q^TQ=I$ and $R$ is an upper triangular matrix.  Then we let $\tilde{R}$ denote the upper nonzeros of $R$, and $\tilde{Q}$ is the
matrix with the corresponding columns of $Q$. Then $\tilde{R}\alpha = \tilde{Q}^TF_t$ and $\alpha = \tilde{R}^{-1}\tilde{Q}^TF_t$. Hence, we have
\begin{align}
\|\alpha\| = \|\tilde{R}^{-1}\tilde{Q}^TF_t\| \leq \|\tilde{R}^{-1}\|\|\tilde{Q}^TF_t\|
\leq \|\tilde{R}^{-1}\|\|{Q}^TF_t\| \leq 2\kappa/(L+\mu), \label{eq:theta}
\end{align}
where (\ref{eq:theta}) uses $F_t=-\frac{2}{L+\mu}\nabla f(x_t)$ and $\tilde{\kappa}=\|\nabla f(x_t)\|/\tilde{\mu}$ (where $\tilde{\mu}$ denotes the least non-zero singular value of $\tilde{R}$).
The proof for Theorem \ref{thm:gel} is finished by plugging (\ref{eq:theta}) into (\ref{eq:finalxx}).
\QEDB

\subsection{Proof of Lemma \ref{lm:boundft1}}
\label{app:pflm2}
First, we obtain the relation between $F_{t+1}$ and $F_1$ by using the Singular Value Decomposition (SVD) for the small matrix $B_t$ for all $t$.

Concretely, Let $\alpha_0^t= 1-\sum_{i=1}^{m_t}{\alpha_i^t}$. Recall that $B_t$ denotes $[F_{t}-F_{t-1}, \dotsc, F_t-F_{t-m_t}]$, i.e.~a matrix with column vectors are $F_t-F_{t-i}$ for $1\leq i\leq m_t$.
Then we adopt the SVD of $B_t$ as $\tilde{U}_t\Sigma_t \tilde{V}^T_t$, where $\tilde{U}^T_t\tilde{U}_t=I$, $\tilde{V}^T_t\tilde{V}_t=I$ and $\Sigma_t=\mathrm{\mathbf{diag}}(\sigma_1,\ldots,\sigma_r)$, $r=\mathrm{\mathbf{rank}}(B_t)$. Then $B^\dag_t = \tilde{V}_t\Sigma_t^{-1} \tilde{U}^T_t $.
Although $B_t$ may have dependent columns, one solution for (\ref{eq:transt}) is that $\alpha^*=B^\dag_tF_t$, where $\alpha^*=(\alpha_1^t,\ldots,\alpha_{m_t}^t)^T$ is the vector of coefficients in (\ref{eq:transt}). Therefore, $\sum_{i=0}^{m_t}{\alpha_i^t F_{t-i}}$ can be represented as $F_t-B_tB^\dag_tF_t=F_t-\tilde{U}_t\tilde{U}^T_tF_t$.

Let $P_t=I-\tilde{U}_t\tilde{U}^T_t$. The matrix $P_t$ is a projection matrix since $P_tP_t=(I-\tilde{U}_t\tilde{U}^T_t)(I-\tilde{U}_t\tilde{U}^T_t)=I-\tilde{U}_t\tilde{U}^T_t$. Also $P_t$ is symmetric. So we finally have $F_{t+1}=(I-\beta_tA)P_tF_t$. Expanding $F_t$ recursively, we get the following relation
\begin{equation}
F_{t+1}=(I-\beta_tA)P_t\dotsm (I-\beta_1A)P_1F_1. \label{eq:expand2}
\end{equation}
We can further have $\|P_j\|_2\le 1$, for $1\le j\le t$. This is due to the following fact
\begin{align*}
\|P_jx\|_2^2=(P_jx)^T(P_jx)&=x^TP_j^TP_jx =x^TP_jx\le \|x\|_2\|P_jx\|_2.
\end{align*}
As $A\in \mathcal{S}_{++}^d$, $A=Q\Lambda Q^T$, where $Q^TQ=I$, and $\Lambda =\mathrm{\mathbf{diag}}(\lambda_1,\dotsc,\lambda_n)$ ($\lambda_j$'s are the real eigenvalues of $A$).

Now, we need to bound $\|F_{t+1}\|_2$. According to (\ref{eq:expand2}), we have
\begin{equation}
\begin{split}\label{eq:grad}
\|F_{t+1}\|_2&=\|(I-\beta_tA)P_t\dotsm (I-\beta_1A)P_1F_1\|_2\\
&\leq\|(I-\beta_tA)P_t\dotsm (I-\beta_1A)P_1\|_2\|F_1\|_2.
\end{split}
\end{equation}
It is sufficient to bound $\|(I-\beta_tA)P_t\dotsm (I-\beta_1A)P_1\|_2$, which is
\begin{equation}
\sup_{\|x\|_2=1}\|(I-\beta_tA)P_t\dotsm (I-\beta_1A)P_1x\|_2. \label{eq:bound}
\end{equation}

Denote the column vectors of $Q$ as $v_1,\dotsc,v_d$ (they are the eigenvectors of $A$). The vector $x$ can be represented as $\sum_{j=1}^dc_{0,j}v_j$, for some $c_{0,j}$'s with $\sum_{j=1}^d c^2_{0,j}=1$. Hence $P_1x$ can be represented as $P_1x=\sum_{j=1}^{n}c_{1,j}v_j$. As $\|P_1\|_2\le 1$, the $c_{1,j}$'s satisfy $\sum_{j=1}^d c^2_{1,j}\le 1$. With $P_1x$, we know $(I-\beta_1A)P_1x=\sum_{j=1}^dc_{1,j}(1-\beta_1\lambda_j)v_j$, where $\sum_j c^2_{1,j}\le 1$. Iteratively expanding, we get $(I-\beta_tA)P_t\dotsm (I-\beta_1A)P_1x= \sum_{j=1}^{d}c_{t,j}(1-\beta_t\lambda_j)\dotsm (1-\beta_1\lambda_j)v_j$, where $\sum_j c^2_{t,j}\le (1+\frac{1}{\sqrt{\kappa}+1})^t$. Hence we have
\begin{equation}\label{eq:chebsp}
(\ref{eq:bound})\le \Big(1+\frac{1}{\sqrt{\kappa}+1}\Big)^t\min_{\beta}\max_{\lambda\in \text{sp}(A)}|H_t(\lambda)|,
\end{equation}
where $H_t(\lambda) =(1-\beta_t\lambda)\dotsm(1-\beta_1\lambda)$ is a degree $t$ polynomial and the $\text{sp}(A)$ is the eigenvalue spectrum of $A$.
As in general, the eigenvalues of $A$ is unknown. We look for the bound of the following form (\ref{eq:chebh2}) instead of (\ref{eq:chebsp}),
\begin{equation}\label{eq:chebh2}
(\ref{eq:chebsp})\leq \Big(1+\frac{1}{\sqrt{\kappa}+1}\Big)^t\min_{\beta}\max_{\lambda\in [\mu,L]}|H_t(\lambda)|.
\end{equation}
Finally, combining (\ref{eq:grad}), (\ref{eq:bound}), (\ref{eq:chebsp}), (\ref{eq:chebh2}), (\ref{eq:ptx}), (\ref{eq:thx1}) and the fact
\[
\Big(1+\frac{1}{\sqrt{\kappa}+1}\Big)^t\Big(1-\frac{2}{\sqrt{\kappa}+1}\Big)^{t/2}\leq 1,
\]
we finish the proof, i.e.,
\[
\|F_{t+1}\|_2/\|F_1\|_2 \leq \sqrt{2\min_{\beta}\max_{\lambda\in [\mu,L]}|H_t(\lambda)|}.
\]
\QEDB

\subsection{Proof of Theorem \ref{thm:guess}}\label{app:pfthm3}

Before to prove Theorem \ref{thm:guess}, we need the following three lemmas.
\begin{lemma}
	\label{lm:boundsum}
	If $\sum_{j=1}^k{e^{i_j}}=e^{i_1}+e^{i_2}+\ldots+e^{i_k}=T$,
	then $\sum_{j=1}^k{i_j}\leq k\ln\frac{T}{k}$.
\end{lemma}
\begin{proof}
	Let $g(x)=e^x$. Note that $g(x)$ is a convex function. According to Jensen's inequality, the following holds.
	\begin{equation*}
	g(\mathbb{E}[x])=\exp\Big(\frac{1}{k}\sum_{j=1}^k{i_j}\Big)\leq \mathbb{E}[g(x)]=\frac{1}{k}\sum_{j=1}^k{\exp(i_j)}
	\end{equation*}
	We obtain $\sum_{j=1}^k{i_j}\leq k\ln\frac{T}{k}$ by taking $\log$ for both sides.
\end{proof}

\begin{lemma}
	\label{lm:boundt}
	Let $T=c\left(\sqrt{\kappa}\ln\frac{1}{\epsilon}+\sqrt{\kappa}(\ln\kappa\ln B)^2\right)$, where $c>2$, then $\sqrt{\kappa}\ln\frac{1}{\epsilon}+\sqrt{\kappa}(\ln\kappa\ln B)\ln\frac{T}{\ln\kappa\ln B} \leq T$ is satisfied.
\end{lemma}
\begin{proof}
	We divide this proof into three cases.
	\begin{enumerate}
		\item $\ln\frac{1}{\epsilon} \leq \ln\kappa\ln B$.\\
		The left-hand side (LHS) of the constraint inequality in this lemma is deduced as follows:
		\begin{equation*}
		\begin{split}
		\sqrt{\kappa}&\ln\frac{1}{\epsilon}+\sqrt{\kappa}(\ln\kappa\ln B)\ln\frac{T}{\ln\kappa\ln B} \\
		\leq\sqrt{\kappa}&\ln\frac{1}{\epsilon}+\sqrt{\kappa}(\ln\kappa\ln B)\ln\frac{2c\sqrt{\kappa}(\ln\kappa\ln B)^2}{\ln\kappa\ln B} \\
		=\sqrt{\kappa}&\ln\frac{1}{\epsilon} +\sqrt{\kappa}(\ln\kappa\ln B)\left(\ln \sqrt{\kappa} + \ln(\ln\kappa\ln B)+\ln2c\right)\\
		\end{split}
		\end{equation*}
		Hence $c\geq 2$ is enough for satisfying LHS $\leq T$.
		
		\item $\ln\frac{1}{\epsilon} > \ln\kappa\ln B >\ln\ln\frac{1}{\epsilon}$.\\
		We also deduce the LHS of the constraint inequality as follows:
		\begin{equation*}
		\begin{split}
		&\sqrt{\kappa}\ln\frac{1}{\epsilon}+\sqrt{\kappa}(\ln\kappa\ln B)\ln\frac{T}{\ln\kappa\ln B} \\
		\leq&\sqrt{\kappa}\ln\frac{1}{\epsilon}+\sqrt{\kappa}(\ln\kappa\ln B)\ln\frac{c(\sqrt{\kappa}\ln\frac{1}{\epsilon})\left(1+\ln\kappa\ln B\right)}{\ln\kappa\ln B} \\  
		\leq&\sqrt{\kappa}\ln\frac{1}{\epsilon}+\sqrt{\kappa}(\ln\kappa\ln B)\ln\left(2c\sqrt{\kappa}\ln\frac{1}{\epsilon}\right) \\
		=&\sqrt{\kappa}\ln\frac{1}{\epsilon}+\sqrt{\kappa}(\ln\kappa\ln B)\Big(\frac{1}{2}\ln\kappa + \ln\ln\frac{1}{\epsilon}+\ln2c\Big)\\
		\leq&\sqrt{\kappa}\ln\frac{1}{\epsilon}+\sqrt{\kappa}(\ln\kappa\ln B)\Big(\frac{1}{2}\ln\kappa + \ln\kappa\ln B+\ln2c\Big)
		\end{split}
		\end{equation*}
		Hence $c\geq 2$ is also enough for satisfying LHS $\leq T=c\left(\sqrt{\kappa}\ln\frac{1}{\epsilon}+\sqrt{\kappa}(\ln\kappa\ln B)^2\right)$.
		
		\item $\ln\kappa\ln B \leq\ln\ln\frac{1}{\epsilon}$.\\
		We deduce the LHS of the constraint inequality as follows:
		\begin{equation*}
		\begin{split}
		&\sqrt{\kappa}\ln\frac{1}{\epsilon}+\sqrt{\kappa}(\ln\kappa\ln B)\ln\frac{T}{\ln\kappa\ln B} \\
		\leq&\sqrt{\kappa}\ln\frac{1}{\epsilon}+\sqrt{\kappa}(\ln\kappa\ln B)\ln\frac{c\sqrt{\kappa}\left(\ln\frac{1}{\epsilon}+(\ln\ln\frac{1}{\epsilon})^2\right)}{\ln\kappa\ln B} \\
		\leq&\sqrt{\kappa}\ln\frac{1}{\epsilon}+\sqrt{\kappa}(\ln\kappa\ln B)\ln c\sqrt{\kappa}\Big(\ln\frac{1}{\epsilon}+\bigl(\ln\ln\frac{1}{\epsilon}\bigr)^2\Big)\\
		\leq&\sqrt{\kappa}\ln\frac{1}{\epsilon}+\sqrt{\kappa}(\ln\kappa\ln B)\left(\frac{1}{2}\ln\kappa\right) + \sqrt{\kappa}\ln\ln\frac{1}{\epsilon}\left(\ln\ln\frac{1}{\epsilon}+\ln c +2\ln\left(\ln\ln\frac{1}{\epsilon}\right)\right)\\
		\end{split}
		\end{equation*}
		Since $\ln(1/\epsilon)>(\ln\ln\frac{1}{\epsilon})^2$ if $(1/\epsilon)>e^2$. Hence it shows that $c\geq 2$ is enough for satisfying LHS $\leq T=c\left(\sqrt{\kappa}\ln\frac{1}{\epsilon}+\sqrt{\kappa}(\ln\kappa\ln B)^2\right)$.
	\end{enumerate}
\end{proof}

\begin{lemma}
	\label{lm:condnum}
	The condition number $\kappa_i$ (in Line 4 of Algorithm \ref{alg:guess}) is always less than $e^2\kappa$, where $\kappa$ is the true condition number. Equivalently, $i$ (in Line 3 of Algorithm \ref{alg:guess}) is always less than $\ln\kappa$.
\end{lemma}
\begin{proof}
	Without loss of generality, let $e^c\leq\mu\leq e^{c+1}$ and $e^d\leq L\leq e^{d+1}$. When $\kappa_i=e^2\kappa$ and $j=e^c$, then $[\mu,L]\subset[\mu_i,L_i]$. According to inequality $\|\nabla f(x_{t+1})\|_2
	\leq 2\big(\frac{\sqrt{\kappa}-1}{\sqrt{\kappa}+1}\big)^t \|\nabla f(x_1)\|_2$ (see the end of the proof of Theorem \ref{thm:opt}), the condition of do-while loop in Line 7--14 of Algorithm \ref{alg:guess} always hold. The only way to break the loop is that the iteration $t>T$, i.e., the end of the algorithm.
\end{proof}

\begin{proofof}{Theorem \ref{thm:guess}}
	According to Lemma \ref{lm:condnum}, $i$ (in Line 3 of Algorithm \ref{alg:guess}) is less than $\ln\kappa$ and $\kappa_i$ is less than $e^2\kappa$. The inner loop $j$ (in Line 5) is obviously less than $\ln B$. Let $k=\ln\kappa\ln B$ and $i_j$ denote the times of the execution of do-while loop (Line 7--14) in each loop iteration (Line 5--16). Thus, the total number of iterations (corresponding to $t$) is $e^{i_j}$ in each loop iteration. These $i_j$ iterations satisfy the do-while condition, i.e., $\frac{\|\nabla f(x_t)\|}{\|\nabla f(x_{t-1})\|}\leq 2\left(\frac{\sqrt{\kappa_i}-1}{\sqrt{\kappa_i}+1}\right)^{t_i}$. We combine the condition together to obtain $\|\nabla f(x_t)\|\leq 2^{i_j}\left(\frac{\sqrt{\kappa_i}-1}{\sqrt{\kappa_i}+1}\right)^{e^{i_j}}\|\nabla f(x_{t-e^{i_j}})\|$. Finally, this guessing algorithm satisfied the following Inequality (\ref{eq:final}).
	
	Note that the Line 15 and 16 of Algorithm \ref{alg:guess} ignore the failed iterations. Also this ignored step can be executed at most once in each loop iteration (Line 5--16). Let $T$ denote the total number of iterations of Algorithm \ref{alg:guess}. Then $\sum_{j=1}^k{e^{i_j}}\leq T\leq 2\sum_{j=1}^k{e^{i_j}} + e\ln\kappa\ln B$.
	\begin{equation}\label{eq:final}
	\|\nabla f(x_t)\|\leq 2^{\sum_{j=1}^k{i_j}}\biggl(\frac{\sqrt{e^2\kappa}-1}{\sqrt{e^2\kappa}+1}\biggr)^{\sum_{j=1}^k{e^{i_j}}}\|\nabla f(x_0)\|
	\end{equation}
	As $\kappa_i$ is less than $e^2\kappa$ and $k=\ln\kappa\ln B$.
	In order to prove the convergence rate, we need the RHS of (\ref{eq:final}) $\leq \epsilon$, it is sufficient to satisfy the following inequality
	\begin{align*}
	\sum_{j=1}^k{i_j} \leq\frac{2}{\sqrt{e^2\kappa}+1}
	\bigg({\sum_{j=1}^k{e^{i_j}}}-\frac{e\sqrt{\kappa}+1}{2}\ln\frac{1}{\epsilon}\bigg),
	\end{align*}
	i.e.,
	\begin{align}
	\frac{e\sqrt{\kappa}+1}{2}\ln\frac{1}{\epsilon} +\frac{e\sqrt{\kappa}+1}{2}\sum_{j=1}^k{i_j}
	\leq \sum_{j=1}^k{e^{i_j}}. \label{eq:last}
	\end{align}
	By applying Lemma \ref{lm:boundsum} and ignoring the constant, we can transform (\ref{eq:last}) to (\ref{eq:idonot}).
	Recall that $\sum_{j=1}^k{e^{i_j}}\leq T\leq 2\sum_{j=1}^k{e^{i_j}} + e\ln\kappa\ln B$ and $k=\ln\kappa\ln B$.
	\begin{equation}\label{eq:idonot}
	\sqrt{\kappa}\ln\frac{1}{\epsilon}+\sqrt{\kappa}(\ln\kappa\ln B)\ln\frac{T}{\ln\kappa\ln B} \leq T.
	\end{equation}
	This is exactly the same as Lemma \ref{lm:boundt}. Thus the proof is finished by using Lemma \ref{lm:boundt}, i.e., $T$ is bounded by $O\left(\sqrt{\kappa}\ln\frac{1}{\epsilon}+\sqrt{\kappa}(\ln\kappa\ln B)^2\right)$.
\end{proofof}

\end{document}